\documentclass[12pt,reqno]{amsart}

\pdfoutput=1

\usepackage{amssymb}
\usepackage{microtype}
\usepackage{subfigure}
\usepackage{tikz}
\usetikzlibrary{arrows}
\tikzstyle directed=[postaction={decorate,decoration={markings,
    mark=at position .65 with {\arrow[arrowstyle]{stealth}}}}]

\usetikzlibrary{arrows,shapes,positioning}
\usetikzlibrary{decorations.markings}
\tikzstyle arrowstyle=[scale=1]

\usepackage{booktabs}
\usepackage[pdftex,colorlinks,citecolor=black,linkcolor=black,urlcolor=black,bookmarks=false]{hyperref}

\usepackage{bbm}

\newtheorem{theorem}{Theorem}[section]
\newtheorem{proposition}[theorem]{Proposition}
\newtheorem{lemma}[theorem]{Lemma}
\newtheorem{corollary}[theorem]{Corollary}

\theoremstyle{definition}

\numberwithin{figure}{section}
\numberwithin{equation}{section}
\numberwithin{table}{section}

\newcommand{\Z}{\mathbb{Z}}
\newcommand{\R}{\mathbb{R}}

\newcommand{\E}{\mathbb{E}}
\newcommand{\prob}{\mathbb{P}}

\newcommand{\vol}{\operatorname{vol}}

\newcommand{\supp}{\operatorname{supp}}

\title[Monochromatic connectivity measures]{Effective bounds for monochromatic connectivity measures in two dimensions}

\author{Matthew de Courcy-Ireland}
\address{
Institute of Mathematics \\
EPFL SB MATH \\
MA C2 647 (B\^atiment MA) \\
Station 8 \\
CH-1015 Lausanne \\
Switzerland} 
\email{matthew.decourcy-ireland@epfl.ch}

\author{Suresh Eswarathasan}
\address{Dalhousie University \\
Department of Mathematics $\&$ Statistics, Chase Building  \\
Coburg Road \\
Halifax, Nova Scotia \\
Canada}
\email{sr766936@dal.ca}

\date{May 13, 2020}

\begin{document}

\begin{abstract}
We establish numerical lower bounds for the monochromatic connectivity measure in two dimensions introduced by Sarnak and Wigman. This measure dictates among the nodal domains of a random plane wave what proportion have any given number of holes, and how they are nested.
Our bounds provide the first effective estimate for the number of simply connected domains and for those that contain a single hole. The deterministic aspect of the proof is to find a single function with a prescribed zero set and, using a quantitative form of the implicit function theorem, to argue that the same configuration occurs in the zero set of any sufficiently close approximation to this function.
The probabilistic aspect is to quantify the likelihood of a random wave being close enough to this function.
\end{abstract}

\maketitle

\section{Introduction}

The Gaussian \emph{random plane wave} was proposed in a celebrated paper of Berry \cite{Berry77} as a model for high-frequency eigenfunctions in classically chaotic systems. 
The \emph{monochromatic connectivity measure} $\mu$ introduced by Sarnak and Wigman \cite{SW18} measures the fraction of nodal domains of a random plane wave with any given number of holes (the number of holes being one less than the number of connected components of the boundary of the domain), with even more refined measures describing the nesting between nodal domains. Sarnak and Wigman proved that $\mu(h) > 0$, for each $h \in \mathbb{Z}_{\geq 0}$, so that each topological type represents a positive proportion of the total number of nodal domains of the random wave.
However, they did not give a quantitative lower bound on $\mu(h)$ for any $h$.  In this note, we do so in the simplest cases of $\mu(0)$ and $\mu(1)$, which correspond respectively to simply connected domains and to domains with a single hole. 

\begin{theorem} \label{thm:main_0}
The monochromatic connectivity measure $\mu$ obeys
\begin{equation}
\mu(0) \geq \frac{1}{c_{{\rm NS}} } \frac{1}{(j_{0,1} + \delta)^2 \sqrt{6 \pi}} 
 \int_{\sqrt{\pi}S/\varepsilon}^{\infty} \left( 1 - \frac{\sqrt{\pi} S}{\varepsilon x} \right) e^{-x^2/2} dx
\end{equation}
for all sufficiently small $\delta > 0$ and $\varepsilon > 0$,
where $S = S(\delta)$ is given by
\begin{equation*}
S = \sum_{n =1}^{\infty} \left( \sup_{B(j_{0,1} + \delta)} |J_n(r)| + \sup_{B(j_{0,1} + \delta)}  |J_n'(r) | + n \sup_{B(j_{0,1} + \delta)}  \left| \frac{J_n(r) }{r} \right| \right),
\end{equation*} 
and $j_{0,1} = 2.4048\ldots$ is the first zero of the Bessel function $J_0$.
\end{theorem}

\begin{theorem} \label{thm:main_1}
Let $S$ be as in the statement of Theorem \ref{thm:main_0} with $j_{0,1}$ replaced by the second zero $j_{0,2} = 5.5200\ldots$ of $J_0$. 
For all sufficiently small $\delta > 0$ and $\varepsilon > 0$,
\begin{equation}
\mu(1) \geq \frac{1}{c_{{\rm NS}} }\frac{1}{(j_{0,2} + \delta)^2 \sqrt{6 \pi}} 
 \int_{\sqrt{\pi} S/\varepsilon}^{\infty} \left( 1 - \frac{\sqrt{\pi} S}{\varepsilon x} \right) e^{-x^2/2} dx
\end{equation}
\end{theorem}

Each theorem is proved by finding a deterministic function with a specified nodal topology, and estimating the probability that a random wave will be close enough 
 to this function for its zero set to also enjoy that topology. 
In both cases, the target is the Bessel function $J_0(r)$ in polar coordinates, as in the \emph{barrier method} initiated by Nazarov and Sodin \cite{NS2009}. In their argument, it was sufficient for the approximation to force a sign change to conclude that there must be a nodal domain in some region. To control the topology of the nodal domain, a more quantitative approximation in the $C^1$ norm is needed. 
The contributions of this article are to make fully explicit how strong an approximation is enough, and how likely such an approximation is to occur. For generalizations and further applications of the barrier method, see \cite{NS2016}, \cite{So12}, \cite{GW14}, and \cite{LL15}. 
 
The factor $c_{ \rm NS }$ appearing in Theorems~\ref{thm:main_0} and \ref{thm:main_1} is the \emph{Nazarov-Sodin constant} for the monochromatic ensemble. It represents the total number of nodal domains per unit volume, to which we compare the number of domains with a particular topology.
An upper bound for $c_{{\rm NS}}$ is needed to deduce a lower bound for $\mu(0)$ and $\mu(1)$.
A simple one can be extracted from the critical points of $F$, noting that $F$ is almost surely smooth and must achieve a local maximum or minimum inside each of its nodal domains.  The expected number of critical points can be computed using the Kac-Rice formula, giving the following result:
\begin{equation} \label{e:cNS_ub}
c_{\rm NS} \leq \frac{1}{2 \pi \sqrt{3}} =0.091888 \ldots
\end{equation}
For details, see Beliaev-Cammarota-Wigman \cite[Proposition 1.1]{BCW19}, Nicolaescu \cite{Nic15}, or Nastasescu \cite[Theorem 3.1]{N11}.

The role of the parameters $\delta$ and $\varepsilon$ is that if two functions differ by at most $\varepsilon$ in the $C^1$ norm on some region, then they must have nodal lines within a distance $\delta$ of each other.
The precise meaning of `sufficiently small' allowed in Theorems~\ref{thm:main_0} and \ref{thm:main_1} is given in Propositions~\ref{prop:stay} and \ref{prop:stay1}. 
We choose $\delta=1/2$ and the largest $\varepsilon$ allowed in each case, roughly $\varepsilon \approx 1/20$.
These choices, together with (\ref{e:cNS_ub}), lead to the following corollary.

\begin{corollary} \label{c:numerical_bounds}
$\mu(0) > 10^{-1282}$ and $\mu(1) > 10^{-4535}$.
\end{corollary}
Corollary~\ref{c:numerical_bounds} gives the first rigorous numerical lower bounds for any atoms of the measure $\mu$. These are still very conservative underestimates of the values
\[
\mu(0) \approx 0.9117 \qquad \text{and} \qquad \mu(1) \approx 0.0514
\]
suggested by simulations of Barnett and Jin \cite{BarJin} (reported in \cite[\S 1.4]{SW18}).
There are two main reasons for the discrepancy, which we discuss further in Section~\ref{sec:conc}. Deterministically, a given nodal topology can easily occur even if the wave is not particularly close to the target $J_0(r)$. This aspect could be improved by incorporating multiple targets. Probabilistically, we have input only the simplest tools in our estimate of how likely the approximation is to occur. Using deeper results on suprema of Gaussian processes would lead to improved lower bounds.
To begin, we have opted for simplicity. 

Underpinning our method for Theorems~\ref{thm:main_0} and \ref{thm:main_1} is an effective form of the implicit function theorem, similar to one used by Cohn-Kumar-Minton \cite{CKM} in the context of sphere packing. Our version Lemma~\ref{lem:scalar} is a simple and sharp criterion for a non-linear equation to have a solution. We hope it will be useful in contexts beyond our own. 
In Section~\ref{sec:symmetrization}, we obtain another bound by a symmetrization method adapted from \cite{IR19}, which Ingremeau-Rivera used to give a lower bound for $c_{{\rm NS}}$.
It is more efficient numerically than Theorems~\ref{thm:main_0} and \ref{thm:main_1}, but relies more heavily on the simple topologies represented by $\mu(0)$ and $\mu(1)$, and on special properties of the monochromatic random wave that would not generalize to other ensembles. 

\begin{theorem} \label{thm:sym}
For any $T > 0$ and $j_{0,1} < r < j_{0,2}$, 
\begin{equation}
\mu(0) \geq \frac{2}{\sqrt{12} c_{{\rm NS}}} r^{-2} \frac{1}{\sqrt{2\pi}} \int_{T}^{\infty} \left(1 - \frac{\sqrt{2}}{2} \frac{r}{\sqrt{1-J_0(r)^2}} \exp\left( \frac{-t^2 J_0(r)^2}{2(1-J_0(r)^2)} \right) \right) e^{-t^2/2} dt.
\end{equation}
Consider $r_1 < \ldots < r_M$ satisfying $j_{0,1} < r_1 < \sqrt{2} j_{0,1}$, $j_{0,2} < r_M < j_{0,3}$, and $r_k^2 - r_{k-1}^2 < j_{0,1}^2$ for $2 \leq k \leq M$.
For any such parameters, and any $T > 0$,
\begin{equation}
\mu(1) \geq \sqrt{2\pi} r_M^{-2}  \int_T^{\infty} \left( 1 - \frac{\sqrt{2}}{2} \sum_k  \frac{r_k}{\sqrt{1-J_0(r_k)^2} } \exp\left( \frac{-t^2 J_0(r_k)^2 }{2(1-J_0(r_k)^2)}\right) \right) e^{-t^2/2} dt.
\end{equation}
\end{theorem}

\begin{corollary} \label{cor:sym-numbers}
$\mu(0) > 10^{-5}$ and $\mu(1) > 10^{-247}$.
\end{corollary}
The atom $\mu(0)$ is significantly easier to treat than $\mu(h)$ for $h>0$ since any nodal domain contains a simply connected domain nested within. For $\mu(1)$, the additional constraints one must impose to guarantee the correct topology lead to a smaller lower bound.

We work directly with a series representation (\ref{e:plane_wave}) for the random wave, in which the main term is $\xi_0 J_0(r)$ where $\xi_0$ is a standard Gaussian, plus another term that will be negligible when $\xi_0$ is large enough. Figure~\ref{fig:barriers} illustrates the effect of increasing $\xi_0$, which makes the nodal lines of $F$ more and more similar to those of $J_0(r)$. In both of our approaches, the goal is to quantify how large $\xi_0$ must be. The method of uniform approximation as in Theorems~\ref{thm:main_0} and \ref{thm:main_1} applies roughly when $|\xi_0| > S/\varepsilon$, whereas the symmetrization method leading to Theorem~\ref{thm:sym} requires $|\xi_0 | > T$. Numerically, for our ultimate choice of parameters, the former requires $|\xi_0| > 43.2831 $  for $\mu(0)$ and $|\xi_0| > 81.4845 $ for $\mu(1)$, compared to $|\xi_0| > 3.2087 $ or $|\xi_0| > 41.9287$ for the latter.

\begin{figure}[t] \label{fig:barriers}
\centering
\subfigure[$\xi_0=0$]{\includegraphics[width=0.3\textwidth]{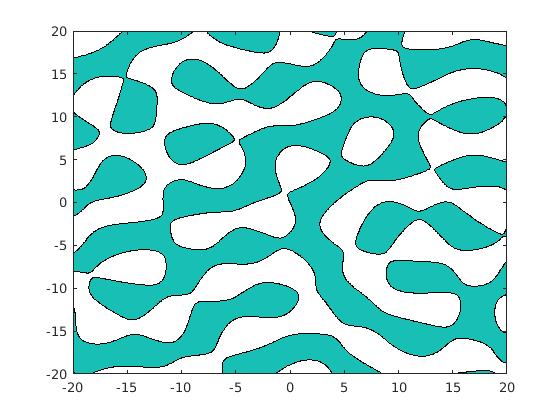}}
\hfill
\subfigure[$\xi_0=3$]{\includegraphics[width=0.3\textwidth]{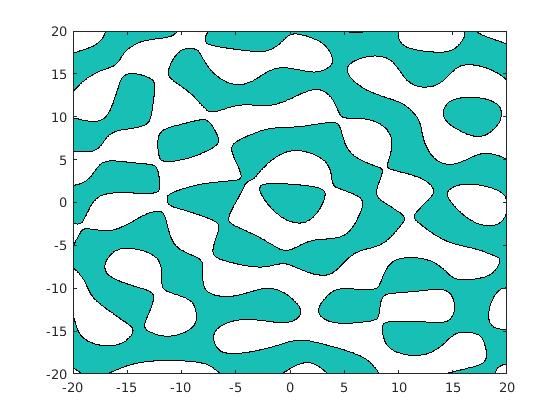}}
\hfill
\subfigure[$\xi_0=5$]{\includegraphics[width=0.3\textwidth]{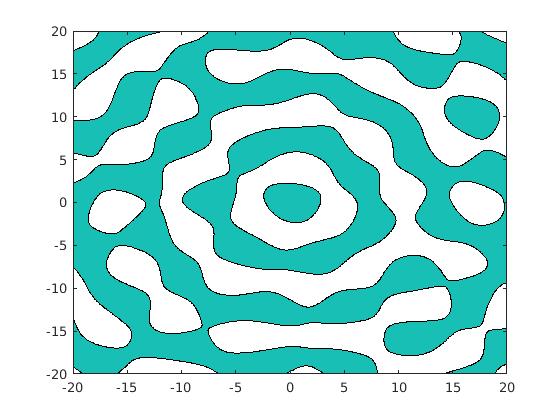}}
\label{fig:green}
\caption{Nodal domains of $\xi_0 J_0(r) + \sum_n (\xi_n \cos{n\theta} + \eta_n \sin{n \theta} )J_n(r)$ where $\xi_0 = 0,3,5$ from left to right and the other coefficients are a sample of random Gaussians (the same sample in each vignette). As $\xi_0$ increases, the nodal lines approximate the concentric circles on which $J_0(r)$ vanishes. The sum is truncated to 100 terms and the functions are evaluated at $500 \times 500$ points in the box $[-20,20] \times [-20,20]$. 
}
\end{figure}

There are other ways to work with the monochromatic random wave, in particular by manipulating its correlation function or viewing it as a Fourier transform of uniform noise on the unit circle. See \cite{SW18} for the latter and \cite{CS18} for the use of more general machinery leading to a proof that the higher-dimensional counterparts of the connectivity measures have full support among all possible topological types. Also in higher dimensions, Enciso and Peralta-Salas have given constructions of eigenfunctions with zero sets containing very general topologies \cite{EPS13}, which can be perturbed to show that those configurations occur with positive probability.
The results we state for connectivity measures on $\R^2$ also apply to similar measures on manifolds, in particular corresponding to random spherical harmonics on $S^2$ or to monochromatic random waves in geometries satisfying a non-self-focal condition. See \cite{CH15} and \cite[Theorem 1.1]{SW18} for more on the scaling argument that relates these measures to the ones in the plane.


In other ensembles of random functions, Gayet-Welschinger \cite{GW14, GW15} have also obtained explicit estimates for topological statistics of zero sets. Instead of Laplace eigenfunctions in the limit of a growing region in $\R^2$ (equivalently, growing eigenvalue on a fixed region), they study sections of high tensor powers of a line bundle on a fixed projective manifold. For example, one can think of homogeneous polynomials of high degree on projective space. The covariance function (defined below) behaves quite differently in these ensembles: for the monochromatic wave, it is oscillatory and decays only slowly, whereas in the more geometric setting it is positive and decays rapidly thanks to the asymptotics of the Bergman kernel.
The constructions one has available for eigenfunctions are also more limited than, say, arbitrary polynomials, which leads to the manipulations with Bessel functions in the present article.

A great inspiration for further work on nodal domains of random waves, and in particular for the foundational paper \cite{NS2009}, was the study of Bogomolny and Schmit \cite{BS02}. They proposed a striking bond percolation model for nodal lines of random Dirichlet eigenfunctions.  An intriguing open question, to which the percolation model suggests an answer, is the behaviour of $\mu(h)$ for large $h$. The data for $h \leq 26$ reported in \cite{SW18} suggest that $\mu(h)$ follows a power law $h^{-\gamma}$ with $\gamma$ slightly larger than 2. This would be consistent with $\gamma = 187/91$, which is the \emph{Fisher exponent} governing the area distribution of clusters in critical percolation.  The Fisher exponent arises in percolation from arguments that have no obvious counterpart for nodal domains, and the area of a domain is only a proxy for its topology, but the data for $h \leq 26$ are consistent with this prediction for $\mu(h)$.  Barnett's webpage \cite{Bar} is an excellent source of software, pictures, and movies related to this circle of problems.

In the remainder of the introduction we describe the random plane wave in more detail, with an emphasis on some fundamental ideas introduced by Nazarov and Sodin; see the Bourbaki article of Anantharaman \cite{A16} for a more in-depth summary.  The connectivity measure $\mu$, appearing in the statements of Theorem~\ref{thm:main_0} and Theorem~\ref{thm:main_1}, is defined more precisely in Section \ref{s:connectivity}.

\subsection{Nodal sets and Gaussian random functions}

\noindent The primary focus of this note is the zero set $\{ F=0 \}$, also referred to as the \emph{nodal set}, of a random function $F: \R^2 \rightarrow \R$. Its connected components are called \emph{nodal lines}, while the connected components of $\{ F \neq 0 \}$ are called \emph{nodal domains}.
The random wave $F$ is a \emph{Gaussian field} in the sense that for any points $x_1, \ldots, x_n$, the random vector $F(x_1),\ldots, F(x_n)$ follows a multivariate normal distribution. The field $F$ is \emph{centered}, meaning $\E[F(x)] = 0$ at each point $x$. 
This singles out the zero set among all the level sets $\{ F = l \}$, $l \in \R$.
For a centered Gaussian, the only information left to specify is the covariance between $F(x)$ and $F(y)$ for each pair of points $x, y \in \R^2$, say
\begin{equation} \label{e:correlation}
K(x,y) = \E[F(x)F(y)].
\end{equation}
This \emph{covariance function} characterizes the statistical distribution of $F$. It is assumed that the law of $F$ is invariant under translations and rotations, so that $K(x,y)$ depends only on $|x-y|$, a manifestation of the \emph{stationarity} property. Finally, the last assumption is that $K(x,x)=1$, that is, the value $F(x)$ at any point is a standard Gaussian of mean zero and unit variance. 
Further relevant information for smooth Gaussian functions, and in particular results concerning spectral functions and covariance functions, can be found in \cite[Appendix A]{NS2016}.

The covariance kernel studied in the present article is
\begin{equation} \label{eq:j0-covariance}
K(x,y) = J_0(|x-y|) = \frac{1}{2\pi} \int_{S^1} e^{i \lambda \cdot (x-y) } d\lambda. 
\end{equation}
This is the covariance function of the random wave model introduced by Berry in \cite{Berry77}. 
The random functions resulting from (\ref{eq:j0-covariance}) are called \emph{monochromatic} because the Fourier transform of $K$ is supported on frequencies $\lambda$ on the unit circle all of the same modulus. 
Almost surely, the resulting function $F$ is smooth and satisfies the Helmholtz equation
\[
(\Delta + 1)F = 0
\]
with $0$ being a regular value.  In particular, $F^{-1}(0)$ is almost surely a disjoint union of simple and smooth curves.  More generally, if $K$ is synthesized using only the frequencies from an algebraic hypersurface $\{ P(\lambda) = 0 \}$, then $F$ will almost surely solve the differential equation
\[
P(i \nabla) F= 0.
\]
In the monochromatic case, $P(\lambda) = \lambda_1^2 + \lambda_2^2 - 1$ and the Fourier transform of $K$ is the uniform measure on the circle $|\lambda| = 1$, normalized to have total mass 1.  In polar coordinates, the monochromatic random wave $F$ takes the explicit form
\begin{equation} \label{e:plane_wave}
F(r,\theta) = \xi_0 J_0(r) +  \sqrt{2} \sum_{n=0}^{\infty} \left( \xi_n J_n(r) \cos(n \theta) +   \eta_n J_n(r) \sin(n \theta) \right)
\end{equation}
where $(\xi_n)_{n=0}^{\infty}$ and $(\eta_n)_{n=1}^{\infty}$ are sequences of independent $N(0,1)$ random variables and $J_n$ is the $n$-th Bessel function of the first kind.
The fact that the random function defined by (\ref{e:plane_wave}) does have covariance given by (\ref{eq:j0-covariance}) is an instance of Neumann's addition theorem for Bessel functions (for which one can refer to \cite[\S 11.2]{W}).  It is convenient to rewrite (\ref{e:plane_wave}) as
 \begin{equation} \label{eq:plane_wave_complex}
F(r,\theta) = \xi_0 J_0(r) + p(r, \theta);
\end{equation}
here, 
\begin{equation} \label{e:p_pert}
p(r, \theta) :=  \sum_{0 \neq n \in \Z} \zeta_n J_{|n|}(r) e^{in\theta}
\end{equation}
where $\zeta_n = (\xi_n - i \eta_n)/\sqrt{2}$ are independent complex Gaussians for $n \geq 1$, while $\zeta_{-n} = \overline{\zeta_n}$. 
This series representation of $F$ is the source of the sums appearing as $S$ in Theorem~\ref{thm:main_0}.

For each $R>0$, the set $B(R)$ denotes the ball with center 0 and radius $R$, and $N(F,R)$ denotes the number of connected components of $F^{-1}(0)$ that are contained entirely inside of $B(R)$.  In \cite{NS2009,NS2016}, Nazarov and Sodin showed that the limit
\begin{equation} \label{e:c_NS}
c_{\rm NS} = \lim_{R \rightarrow \infty} \frac{\mathbb{E} N(F,R)}{\pi R^2}
\end{equation}
exists and is positive.  Its exact value remains mysterious.
Values for $4\pi c_{{\rm NS}}$ are commonly stated, with the factor $4\pi = \vol(S^2)$ leading to a dimensionless quantity instead of a number per unit volume.
The percolation model appearing in the work of Bogomolny and Schmit \cite{BS02} leads to an approximation
\[
c_{{\rm NS}}  \approx \frac{1}{4 \pi} \frac{3\sqrt{3}-5}{\pi} = \frac{1}{4 \pi} \times 0.0624 \dots 
\]
based on previous work of Temperley-Lieb \cite[formula (41), Table 5]{TL71} and Ziff-Finch-Adamchik \cite[equation (7)]{ZFA} .
This is an overestimate compared to the value suggested by simulations, namely $c_{{\rm NS}} \approx \frac{1}{4 \pi} \times .0589 $ first obtained by Nastasescu \cite{N11}; see also the work of Beliaev-Kereta \cite{BK13} and Konrad \cite{K12}.
As mentioned in (\ref{e:cNS_ub}), upper bounds of order $10^{-1}$ are easy to obtain from critical points using the Kac-Rice formula.
Ingremeau-Rivera \cite{IR19} gave a rigorous lower bound
\[
c_{{\rm NS}} \geq \, \frac{1.39 \times 10^{-4}}{4\pi} 
\]
Their method combines a barrier-style comparison to $J_0(|x|)$ with a new use of the Kac-Rice formula, and considerations of volumes of nodal domains. In Section~\ref{sec:symmetrization}, we modify this approach to allow some control over the topology of the resulting nodal domains.

\subsection{Semi-locality of the counting function $N(F,R)$}

Suppose $T$ is a type of nodal domain, usually of a topological nature. For instance, the ``domains of type $T$" might be the ones with any given number of holes. 
Given a type $T$, write $N_T(F, B)$ for the number of nodal domains of $f$ of type $T$ and contained inside a region $B$. For this article, the case that $B$ is a ball with growing diameter $\text{diam}(B) \rightarrow \infty$ is most interesting. It is very difficult to study large nodal domains, for instance nodal lines that cross from one part of $\partial B$ to a distant part. The small nodal domains can be captured to some extent by packing disjoint domains $D_j$ inside $B$. Indeed, if $D_j$ are disjoint domains inside $B$, then
\begin{equation} \label{eq:disjoint}
N_T(F,B) \geq \sum_{j} N_T(F, D_j)
\end{equation}
because there is no overlap between the nodal domains contained in $D_j$ and $D_k$ for $j \neq k$.
Note that only the domains contained entirely within the interior of $D$ are included in the count $N_T(F,D)$, so it is not possible for a domain in $D_j$ to merge with one in $D_k$.
Taking expectations of (\ref{eq:disjoint}) gives
\[
\E[N_T(F,B)] \geq \sum_j \E[N_T(F,D_j)].
\]
For small domains $D_j$, we hope not to lose very much by making the further estimate
\[
\E[N_T(F, D_j) ] = \sum_{k=0}^{\infty} k \prob(N_T(F,D_j)=k) \geq \prob(N_T(F,D_j) \geq 1)
\]
because it is expected that $D_j$ contains at most one nodal domain of the given type. 
This is a manifestation of the \emph{semi-local} nature of nodal domains, meaning that most of the nodal domains are small. A precise form of this key property was established by Nazarov and Sodin and enables their analysis in \cite{NS2009} and \cite{NS2016}. The semi-locality of nodal domains mitigates the kind of global problem illustrated in Figure \ref{fig:enemy}, in which long nodal lines are not captured by the system of local detectors $D_j$.

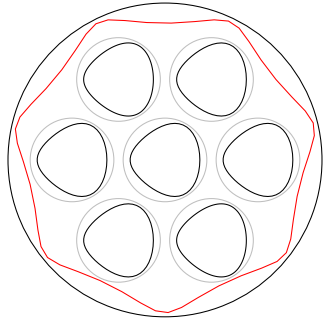
\begin{figure}[h] \label{fig:enemy}
\begin{tikzpicture}[scale=0.618]
\draw (0,0) circle (0.975*2*1.732);
\draw[lightgray] (0,0) circle (0.9);
\draw [samples=50,domain=0:360] plot ({0.75*cos(\x)},{0.125*sin(2*\x)+0.75*sin(\x)});
\foreach \a in {-1,1}
\draw[lightgray] (\a,1.732*\a) circle (0.9);
\foreach \a in {-1,1}
\draw[lightgray] (\a,-1.732*\a) circle (0.9);
\foreach \a in {-1,1}
\draw[lightgray] (-2*\a,0) circle (0.9);
\foreach \a in {-1,1}
\draw [samples=50,domain=0:360] plot ({-2*\a+0.75*cos(\x)},{0.125*sin(2*\x)+0.75*sin(\x)});
\foreach \a in {-1,1}
\draw [samples=50,domain=0:360] plot ({\a+0.75*cos(\x)},{1.732*\a+0.125*sin(2*\x)+0.75*sin(\x)});
\foreach \a in {-1,1}
\draw [samples=100,domain=0:360] plot ({\a+0.75*cos(\x)},{-1.732*\a + 0.125*sin(2*\x)+0.75*sin(\x)});
\draw[red] [samples=50,domain=0:360] plot ({1.8*1.732*cos(\x)+0.1*1.732*sin((6)*\x)},{1.8*1.732*sin(\x)+0.1*1.732*cos((6)*\x)});
\end{tikzpicture}
\caption{The outer circle bounds the ball $B$. The inner grey circles are the domains $D_j$. These contain some small nodal lines shown in black. A larger one shown in red avoids detection as it does not intersect any of the domains $D_j$. Even if it did, a local inspection inside the disks would not determine whether it is a single connected piece or several lines. Nor is it possible to tell, only from the insides of the disks $D_j$, whether the red line is contained in $B$.}
\end{figure}

In the simplest version of the argument, the domains $D_j$ are equal-sized disks as in \cite{NS2009}.
By translation-invariance, they all have an equal chance of containing a nodal domain of the desired type.
It follows that
\[
\E[N_T(F,B)] \geq (\# \ \text{of disks} )\prob(N_T(F,D) \geq 1 )
\]
where $D$ is any disk of that size. 
We normalize by the volume of $B$:
\[
\frac{\E[N_T(F,B)]}{\vol{B}} \geq \prob(N_T(F,D) \geq 1)  \frac{(\# \ \text{of disks} )}{\vol{B}}.
\]
Maximizing the number of disks $D_j$ that can fit inside $B$ is a packing problem.
As $\text{diam}(B) \rightarrow \infty$, the greatest number is given by the packing density in the plane.
Arranging the disks in a hexagonal pattern gives
\[
\frac{(\# \ \text{of disks} )}{\vol{B}} \sim \frac{\pi}{\sqrt{12}} \frac{1}{\vol{D}}
\]
and therefore
\begin{equation} \label{eq:expected-per-volume}
\liminf_{R \rightarrow \infty} \frac{\E[N_T(F,B)]}{\vol{B}} \geq \frac{ \prob(N_T(F,D) \geq 1) }{\vol{D} } \frac{\pi}{\sqrt{12}}.
\end{equation}
Thus a lower bound on the volume density of nodal domains of type $T$ is obtained, provided that one can produce such a domain in a single disk.  Note that one could replace the above $\liminf$ operation with that of $\lim$ for all types $T$ considered in this article.  The existence of this limit follows from the methods in \cite{NS2016}. In particular, \cite{SW18} established this when $T$ refers to domains with a given number of holes.

To bound $\prob( N_T(F,D) \geq 1)$ from below, it suffices to deterministically produce a function with a nodal domain of type $T$ contained in $D$, and then argue that $F$ has a positive probability of approximating this function. The implicit function theorem guarantees that if a function vanishes somewhere and has non-zero gradient, then any other function approximating it closely enough will also vanish somewhere nearby. In particular, provided the coefficient $\xi_0$ is large enough, the random wave $\xi_0 J_0(r) + p(r,\theta)$ will vanish on a curve close to any of the circles on which $J_0(r)=0$.  This is the source of the nodal lines used to prove Theorems~\ref{thm:main_0} and \ref{thm:main_1}. To achieve a numerical lower bound, the implicit relations ``closely enough" and ``nearby" must be made fully explicit.
This is achieved in Section~\ref{sec:newt}.


\subsection{Monochromatic connectivity measures and counting topologies} \label{s:connectivity}

Sarnak and Wigman \cite{SW18} proved a refinement of (\ref{e:c_NS}) towards the counting of nodal domains (or components) of $F$ of a given topological class by utilizing the overall methodology found in \cite{So12}.  The authors also found an elegant formulation of this counting in terms of a probability measure whose support is the space of nesting trees, whose construction is described more below.

We first define the spaces used in building this probability measure.  Let $\Omega(F, R)$ denote the set of nodal domains of $F$ that lie entirely inside of $B(R)$.  This set is the disjoint union of elements $\omega$ which are compact, 2-dimensional manifolds with smooth boundaries.  Furthermore, denote by $\mathcal{C}(F, R)$ the set of connected components of the nodal set of $F$ that lie entirely inside of $B(R)$.  It also follows that $\mathcal{C}(F, R)$ is a disjoint union of elements $c$ which this time are smooth, simple, and closed curves.  

With the sets $\Omega(F,R)$ and $\mathcal{C}(F, R)$ in hand, the \textit{nesting tree} $X(F,R)$ is a graph that captures the nesting relations between the elements $c \in \mathcal{C}(F, R)$ and $\omega \in \Omega(F,R)$ in the following way.  The vertices of $X(F,R)$ correspond to points $\omega \in \Omega(F,R)$ and there exists an edge between $\omega$ and $\omega'$ if these two nodal domains share a unique common boundary component $c \in \mathcal{C}(F, R)$.  That is, the edges of $X(F,R)$ are exactly the elements $c \in \mathcal{C}(F, R)$. 

Now, define a map $e$ (whose symbol is chosen in order abbreviate the word ``tree end") between the set of components $\mathcal{C}(F,R)$ and $\mathcal{T}$ which is the countable set of finite rooted trees.  Given a nodal component $c$, consider the two finite rooted trees that arise as a result of removing the edge in $X(F,R)$ corresponding to $c$.  Let $e(c) \in \mathcal{T}$ be the smaller of the two trees, therefore resembling a nesting tree end for the larger $X(F,R)$.  If the two resulting trees have equal size, then choosing either tree has no effect on the eventual probability space of interest: $(\mathcal{T}, \mu_X)$.

Since $|\mathcal{C}(F,R)| = N(F,R)$ is always finite thanks to the regularity of our fields $F$, it is possible to consider the (random) measure
\begin{equation*}
\mu_{X(F,R)} : = \frac{1}{|\mathcal{C}(F,R)|} \sum_{c \in \mathcal{C}(F,R)} \delta_{e(c)}
\end{equation*}
where $\delta_{e(c)}$ is the Dirac delta function at the point $e(c) \in \mathcal{T}$.  This measure $\mu_{X(F,R)}$ quantifies the distribution of nesting tree ends of $X(F,R)$.  Notice that if $N_X(F,G,R)$ denotes the number of edges of $X(F)$ whose corresponding tree end $e(c)$ is graph-isomorphic to a tree-type $G \in \mathcal{T}$, then one can write
\begin{equation*}
\mu_{X(F,R)} = \frac{1}{|\mathcal{C}(F,R) |} \sum_{G \in \mathcal{T}} N_X(F,G,R) \, \delta_G.
\end{equation*} 

One of the important results leading to the main theorem of Sarnak and Wigman's, namely \cite[Theorem 4.1]{SW18}, states that there exists a (random) probability measure $\mu_{X(F)}$ on $\mathcal{T}$ with $\supp \mu_{X(F)} = \mathcal{T}$.  Here, $\mu_{X(F)}$ is a total-variation limit of the $\mu_{X(F,R)}$ and takes the form of 
\begin{equation*}
\mu_{X(F)} = \frac{1}{ c_{\rm NS}} \sum_{G \in \mathcal{T}} c_{X(F)}(G) \, \delta_G
\end{equation*}
where the random variables $c_{X(F)}(G)$ are the $L^1$ limits in $R$ of $\frac{N_X(F,G,R)}{\pi R^2}$.  But an application of the Wiener Ergodic Theorem \cite[Section 3.3]{SW18} shows that $c_{X(F)}(G)$ is the $L^1$ limit of the deterministic function $\frac{\mathbb{E} \left[ N_X(F,G,R) \right]}{\pi R^2}$.  A simple ``tightness" argument \cite[Appendix B]{SW18} finally gives a \textit{deterministic probability} measure
\begin{equation} \label{e:tree_measure}
\mu_{X} =  \frac{1}{c_{\rm NS}} \sum_{G \in \mathcal{T}} \mathbb{E}\left[ c_{X(F)}(G) \right] \, \delta_G
\end{equation}
on the space $\mathcal{T}$.  Note that given a more general ensemble of Gaussian fields $F$, one can construct an analogous measure $\mu_X$ that is \emph{universal}, that is it depends only $K$.  This universality has a consequence when considering analogous random waves on more general compact surfaces.

It is this measure $\mu_X$ that is called the \textit{monochromatic connectivity measure} in \cite{SW18}.  It is important to note that in the case of $G_0$ being a single vertex or $G_1$ a tree on two vertices, the only possible topologies of planar domains are those with no holes and one hole, respectively.  Thus, considering Theorems \ref{thm:main_0} and \ref{thm:main_1}, one can relabel the atoms of $\mu_X$ as follows: $ \mu_X(G_0)=\mu(0)$ and $ \mu_X(G_1) = \mu(1)$.  Following \cite[Theorems 3.3, 3.4, 4.2]{SW18} and using the Borel-Cantelli Lemma, we find that
\begin{equation} \label{e:density_lb}
\mathbb{E}\left[ c_{X(F)}(G) \right]  \geq \liminf_{R \rightarrow \infty} \, \frac{\mathbb{E} \left[ N_X(F,G,R) \right]}{\pi R^2} .
\end{equation}
Hence, the main results of this article give lower bounds for $\mu(0)$ and $ \mu(1)$.

\section{Use of the Implicit Function Theorem} \label{sec:newt}

The following ``Explicit Function Theorem" is a quantitative version of the usual Implicit Function Theorem.
The statement is inspired by a similar formulation given by Cohn-Kumar-Minton \cite[Theorem~3.1]{CKM}, which also applies to vector-valued functions (that is, to systems of equations instead of a single one). Their proof uses a continuous version of Newton's method, following a similar argument that appears in Neuberger's article \cite[Theorem 2]{N}. In the scalar case, it is possible to give a proof under more flexible hypotheses without using Newton's method.


\begin{lemma} \label{lem:scalar}
Suppose $F: B(x_0,\delta) \rightarrow \R$ is a $C^1$ function defined on a ball in $\R^n$ and $T \in \R^n$ is a non-zero vector.
Assume that for all $x \in B(x_0,\delta)$,
\begin{equation} \label{eq:hypo}
\left| \nabla F(x) \cdot T \right| > \frac{ | T | |F(x_0)| }{\delta}.
\end{equation}
Then $F$ has a zero inside $B(x_0,\delta)$. 
\end{lemma}

We remind the reader that once a single zero has been produced, it then follows from the standard Implicit Function Theorem that $F^{-1}(0)$ is a $C^1$ submanifold of $\R^n$. Moreover, if $F$ is $C^{\infty}$, then $F^{-1}(0)$ is a $C^{\infty}$ submanifold.

Lemma~\ref{lem:scalar} is sharp in the sense that (\ref{eq:hypo}) cannot be relaxed to $|\nabla F(x) \cdot T | > c |T| |F(x_0)| \delta^{-1}$ for any value $c < 1$. For example, the linear function $F(x,y) = x$ does not have a zero inside the open ball $B(x_0,\delta)$ around $x_0 = (\delta,0)$, even though it would satisfy the relaxed hypothesis with $T = (1,0)$.

\begin{proof}[Proof (without Newton's method)]
Let us assume $F(x_0) \neq 0$, for otherwise a zero already exists in $B(x_0, \delta)$.
By the fundamental theorem of calculus,
\begin{equation} \label{e:fund_thm_calc}
\overline{F}(r) : = F\left( x_0 + rU \right) = F(x_0) + \int_0^{r} \nabla F(x_0 + s U) \cdot U ds
\end{equation}
where $U = T/|T|$ is the unit vector in the same direction as $T$ and $r \in [0, \delta)$.
By hypothesis, $\nabla F(x) \cdot U$ is non-zero, and by continuity maintains the same sign  throughout $B(x_0,\delta)$. It follows that
\[
\int_0^{\delta} \nabla F(x_0 + sU) \cdot U ds = \pm \int_0^{\delta} | \nabla F(x_0 + s U) \cdot U | ds.
\]
In particular, by (\ref{eq:hypo}),
\[
\left| \int_0^{\delta} \nabla F(x_0 + sU) \cdot U ds \right| = \int_0^{\delta} | \nabla F(x_0 + sU) \cdot U | ds > |F(x_0)|.
\]

Now suppose the sign of $\nabla F(x) \cdot U$ is opposite that of $F(x_0)$. 
The above estimate, along with (\ref{e:fund_thm_calc}), shows that the 1-dimensional function $\overline{F}(r)$ exhibits a sign change somewhere in $[0, \delta)$.
Therefore $F$ itself must vanish somewhere inside the segment joining $x_0$ to $x_0 + \delta U$.  If the signs of $\nabla F(x) \cdot U$ and $F(x_0)$ are the same, then we instead apply the same argument to $\overline{F}(r) =  F\left( x_0 - r U \right) $. In either case, $F$ has a zero inside the segment joining $x_0$ to $x_0 \pm \delta U$, hence inside $B(x_0,\delta)$ as required.
\end{proof}

For comparison, here is another proof following \cite[Theorem 3.1]{CKM}.  In that reference, $F$ can take values in a normed vector space $W$, not necessarily $\R$, in which case $T: V \rightarrow W$ is a linear operator instead of a single vector. Their hypothesis is that $\| DF(x) \circ T - I \|_{ {\rm op }} < 1 - \| T \|_{ {\rm op} } |F(x_0)|/\delta$, which in the scalar case is somewhat more restrictive than (\ref{eq:hypo}).
It assumes \emph{both} inequalities
\[
\frac{ |T| |F(x_0)| }{\delta} < \nabla F(x) \cdot T < 2 - \frac{ |T| |F(x_0)|}{\delta}
\]
whereas (\ref{eq:hypo}) assumes only the first one (or its negative).
Another difference between the hypotheses is that (\ref{eq:hypo}) is invariant under rescaling $T$.

\begin{proof}[Proof (with Newton's method)]
We produce a curve $x(t)$ starting from $x(0)=x_0$ and such that, for $0 \leq t \leq 1$,
\[
F( x(t) ) = (1-t) F(x_0)
\]
and $x(t) \in B(x_0, \delta)$.  Then $x(1)$ is the required zero of $f$. We obtain $x(t)$ from the differential equation
\begin{equation} \label{eq:newt2}
\frac{dx}{dt} = -T \big( \nabla F(x(t)) \cdot T  \big)^{-1} F(x_0).
\end{equation}
Indeed, if $x(t)$ solves (\ref{eq:newt2}) on $[0,1]$, then the chain rule implies that
\[
\frac{d F(x(t)) }{dt} = \nabla F(x(t)) \cdot \frac{dx}{dt} = - \nabla F(x(t)) \cdot T \big( \nabla F(x(t)) \cdot T)^{-1} F(x_0) = -F(x_0).
\]
Therefore, having the same derivative, $F(x(t))$ and $(1-t)F(x_0)$ differ only by a constant. Since they agree at $t=0$, they must then agree for all $t$.

By hypothesis, $\nabla F(x) \cdot T \neq 0$ for $x \in B(x_0,\delta)$ so that (\ref{eq:newt2}) is well-defined. We claim that it has a solution $x(t)$ defined throughout the interval $0 \leq t \leq 1$ and such that $x(1) \in B(x_0, \delta)$. 
Recall that Peano's existence theorem guarantees the existence of solutions to a differential equation $dx/dt = R(t,x(t))$ provided only that $R$ is continuous. In particular, because $F$ is $C^1$, Peano's theorem applies to (\ref{eq:newt2}). Any solution $x(t)$ satisfies
\[
\left| \frac{dx}{dt} \right| \leq \frac{| T | |F(x_0)| }{ | \nabla F(x(t)) \cdot T | }< \delta
\]
by (\ref{eq:hypo}). By integration, $|x(t) - x_0 | < \delta t$. It follows that $x(t)$ can be extended to $t=1$ while remaining inside $B(x_0,\delta)$.
\end{proof}

Suppose there is a given function $G$ on $B(x_0,\delta)$ such that $G(x_0) = 0$. We would like to use Lemma~\ref{lem:scalar} to produce a zero of any function $F$ approximating $G$ closely enough near $x_0$. In the application to our main theorems, $G$ is $J_0(r)$ and $F$ is $J_0(r) + p(r,\theta)/\xi_0$ with $p$ defined in equation (\ref{e:p_pert}), but this step in the proof goes through more generally.

\begin{proposition} \label{prop:approx}
Suppose $G$ is a $C^1$ function on a ball $B(x_0,\delta)$ in $\R^n$, with $G(x_0) = 0$ and $\nabla G(x_0) \neq 0$.
Assume $F$ is another $C^1$ function on $B(x_0,\delta)$ such that
\begin{equation} \label{eq:hypo2}
\sup_{B(x_0,\delta) } \max( |F - G|, | \nabla F - \nabla G | ) < \frac{\delta}{1+\delta} \inf_{B(x_0, \delta)} \left| \nabla G(x) \cdot \frac{\nabla G(x_0)}{|\nabla G(x_0)|} \right|.
\end{equation}
Then $F$ has a zero in $B(x_0,\delta)$.
\end{proposition}

For the result to be non-trivial (that is, to apply to some functions $F \neq G$), $\delta$ must be small enough that $\nabla G(x) \cdot \nabla G(x_0) \neq 0$ for $|x - x_0 | < \delta$.

\begin{proof}
We apply Lemma~\ref{lem:scalar}, choosing $T = \nabla G(x_0)$. We must confirm (\ref{eq:hypo}), which becomes
\[
| \nabla F(x) \cdot \nabla G(x_0) | > | \nabla G(x_0) | | F(x_0) | \delta^{-1}.
\]
We verify this by comparing $F$ to $G$ as follows. For brevity, write
\[
\varepsilon = \| F - G \| = \sup_{B(x_0,\delta) } \max( |F - G|, | \nabla F - \nabla G | ).
\]
By the triangle inequality and Cauchy-Schwarz,
\begin{align} \label{e:dot_prod_lb}
\nonumber| \nabla F(x) \cdot \nabla G(x_0) | &\geq | \nabla G(x) \cdot \nabla G(x_0) | - |(\nabla F(x) - \nabla G(x) ) \cdot \nabla G(x_0) | \\
&\geq \inf_{B(x_0,\delta)} | \nabla G(x) \cdot \nabla G(x_0) | - \varepsilon | \nabla G(x_0) |.
\end{align}
On the other hand, $|F(x_0)| \leq \| F - G \| = \varepsilon$ because $G(x_0) = 0$.  
Imposing that the rightmost side of (\ref{e:dot_prod_lb}) is greater than or equal to $|\nabla G(x_0) | \, \epsilon \, \delta^{-1}$ and using $|F(x_0)| \leq \varepsilon$ leads us to a sufficient condition for (\ref{eq:hypo}):
\[
 \inf_{B(x_0,\delta)} | \nabla G(x) \cdot \nabla G(x_0) | - \varepsilon | \nabla G(x_0) |
 > \varepsilon \delta^{-1} |\nabla G(x_0) |.
\]
Hence (\ref{eq:hypo}) follows from
\[
\varepsilon <  \frac{\delta}{1+\delta} \inf_{B(x_0, \delta)} \left| \nabla G(x) \cdot \frac{\nabla G(x_0)}{|\nabla G(x_0)|} \right| 
\]
which is precisely what we have assumed in (\ref{eq:hypo2}).
\end{proof}

In our application, $G = J_0(r)$ and the gradient is given by
\[
\nabla G(x) = J_0'(r) (\cos{\theta},\sin{\theta})
\]
where $(\cos{\theta},\sin{\theta})$ is a unit vector in the radial direction and $r ,\theta$ are the polar coordinates of $x$.
It is convenient for numerical purposes to note that $J_0'(r) = -J_1(r)$ is another Bessel function.
We take $x_0 = (j_{0,1},0)$, where $j_{0,1} = 2.4048\ldots$ is the first root of $J_0$, noting by symmetry that the same calculations apply to any other $x_0$ on the circle $|x_0| = j_{0,1}$.
Then
\[
\nabla G(x) \cdot  \frac{ \nabla G(x_0) }{| \nabla G(x_0) |} = J_0'(r) \cos{\theta}.
\]
To apply Proposition~\ref{prop:approx}, we must choose $\delta$ small enough that neither $J_0'(r)$ nor $\cos{\theta}$ vanishes in $B(x_0,\delta)$.
To control $J_0'(r)$, we assume that
\begin{equation} \label{eq:j0'-neq0}
\delta < j_{1,1} - j_{0,1} = 1.42688\ldots
\end{equation}
where $j_{1,1} = 3.8317\ldots$ is the first positive root of $J_1$. 
Note that $j_{1,1}-j_{0,1} < j_{0,1}$, so that (\ref{eq:j0'-neq0}) also implies that $-\delta > - j_{0,1}$.
Since $j_{0,1} - \delta < r < j_{0,1} + \delta$ for points in the ball $B(x_0,\delta)$, it follows that $0 < r < j_{1,1}$. By definition, $J_1$ has no other roots between $0$ and $j_{1,1}$, so $J_0'(r) = -J_1(r) \neq 0$ for $0 < r < j_{1,1}$. Thus (\ref{eq:j0'-neq0}) guarantees that $J_0'(r) \neq 0$ throughout $B(x_0,\delta)$. 

By chance and trigonometry, it happens that (\ref{eq:j0'-neq0}) is also enough to guarantee $\cos{\theta} \neq 0$ on $B(x_0,\delta)$. 
Moreover,
\[
\cos{\theta} > \frac{r^2 + j_{0,1}^2 - \delta^2 }{2j_{0,1} r} \geq \sqrt{1 - \frac{\delta^2}{j_{0,1}^2} } 
\]
with the greatest angle achieved as part of a right triangle of sides $\delta, j_{0,1},$ and $\sqrt{j_{0,1}^2-\delta^2}$. 
In particular, $\cos{\theta} \neq 0$ for $\delta < j_{0,1} = 2.4048\ldots$, and this already follows from (\ref{eq:j0'-neq0}).

\section{The probability of a good approximation} \label{sec:prob}

Given a domain $A$ in $\R^2$, we use the following version of the $C^1$ norm in polar coordinates:
\begin{equation} \label{eq:Anorm}
\| F \|_A = \sup_A \max\left( |F| , \left| \frac{\partial F}{\partial r} \right|, \frac{1}{r} \left| \frac{\partial F}{\partial \theta} \right| \right)
\end{equation}
In our application, $A$ will be an annulus or a ball contained within this annulus.
The norm (\ref{eq:Anorm}) controls $|\nabla F|$ via
\begin{equation} \label{eq:two-norms}
\sup_A | \nabla F | \leq \sqrt{2} \| F \|_A.
\end{equation}

\begin{proposition} \label{prop:prob}
Let $F = \xi_0 J_0(r) + p(r,\theta)$ be a random wave as in (\ref{eq:plane_wave_complex}) and let $A$ be any domain in $\R^2$. 
For any $\varepsilon > 0$,
\begin{equation}
\prob\left( \| p \|_{A} < \varepsilon |\xi_0| \right) \geq \frac{2}{\sqrt{2\pi}} \int_{\sqrt{\pi} S/\varepsilon}^{\infty} \left( 1 - \frac{\sqrt{\pi} S}{\varepsilon x} \right) e^{-x^2/2} dx
\end{equation}
where $S = S(A)$ is given by
\begin{equation} \label{eq:s}
S = \sum_{n =1}^{\infty} \left(\sup_A |J_n(r)| + \sup_A |J_n'(r) | + n \sup_A \left| \frac{J_n(r) }{r} \right| \right).
\end{equation}
\end{proposition}

Proposition~\ref{prop:prob} follows from a tail estimate for $\| p \|_{A}$, which does not require the coefficients of the random wave to be Gaussian. 
Naturally, more powerful tools are available in the Gaussian case, but the following simple argument is sufficient for our purposes and easy to make explicit.
\begin{proposition} \label{prop:tail}
Let $p$ be the random function
\[
p(r,\theta) = \sum_{n \neq 0} \zeta_n J_{|n|}(r) e^{i n \theta}
\]
where $\zeta_n$ are complex-valued random variables satisfying $\E[|\zeta_n|] \leq Z$.
Then, for any $\tau > 0$ and any domain $A$, 
\[
\prob\left( \| p \|_{A} < \tau \right) \geq 1 - 2ZS/\tau 
\]
where $S = S(A)$ is the sum (\ref{eq:s}).
\end{proposition}

\begin{proof}[Proof of Proposition~\ref{prop:tail}]
The $C^1$ norm of $p$, as in (\ref{eq:Anorm}), controls three quantities. A union bound over these three cases leads to
\[
\prob\left( \| p \|_{A } < \tau \right) \geq 1 - \prob\left( \sup_A |p| \geq \tau \right) - \prob\left( \sup_A \left| \frac{\partial p }{\partial r} \right| \geq \tau \right) - \prob\left( \sup_A \left| \frac{1}{r} \frac{\partial p }{\partial \theta} \right| \geq \tau \right) 
\]
for any value of the tolerance $\tau$.
We will bound each of the latter terms using Markov's inequality.
Suppose that $\sup_A |p| \geq \tau$. Then for some values of $r$ and $\theta$,
\[
\tau \leq |p(r,\theta) | \leq \sum_{n \neq 0} |\zeta_n| |J_{|n|}(r) | \leq \sum_{n \neq 0} |\zeta_n| \sup_A |J_{|n|}(r) |.
\]
It follows that
\[
\prob\left( \sup_A |p| \geq \tau \right) \leq \prob\left( \sum_{n\neq 0} |\zeta_n| \sup_A |J_{|n|}(r)| \geq \tau \right).
\]
By Markov's inequality,
\[
\prob\left( \sum_{n\neq 0} |\zeta_n| \sup_A |J_{|n|}(r)| \geq \tau \right) \leq \frac{1}{\tau} \E\left[ \sum_{n \neq 0} |\zeta_n| \sup_{A} |J_{|n|}(r) | \right] \leq \frac{Z}{\tau} \sum_{n \neq 0} \sup_A J_{|n| }(r).
\]
Similar calculations apply to the derivatives of $p$, using the series
\[
\frac{\partial p }{\partial r} = \sum_{n \neq 0} \zeta_n J_{|n|}'(r) e^{in\theta}, \quad \frac{1}{r} \frac{\partial p}{\partial \theta} = \sum_{n \neq 0} \zeta_n \frac{J_{|n|}(r)}{r} i n e^{in\theta}.
\]
Combining the three terms, and merging the equal contributions from $n$ and $-n$, we obtain
\[
\prob\left( \| p \|_{A} < \tau \right) \geq 1 - \frac{2Z}{\tau} \sum_{n =1}^{\infty} \left( \sup_A |J_n(r)| + \sup_A |J_n'(r) | + n \sup_A \left| \frac{J_n(r) }{r} \right| \right)
\]
as required.
\end{proof}

\begin{proof}[Proof of Proposition~\ref{prop:prob}]
Since $\xi_0$ is a standard Gaussian independent of $p$,
\[
\prob\left( \| p \|_{A} < \varepsilon |\xi_0| \right) = \frac{1}{\sqrt{2\pi} }\int_{-\infty}^{\infty} \prob\left( \| p \|_{A} < \varepsilon |x| \right) e^{-x^2/2} dx.
\]
For small values of $\tau$, the lower bound given by Proposition~\ref{prop:tail} might be no better than the trivial bound that probabilities are nonnegative. 
We keep only the large values of $\tau = \varepsilon |x|$ for which Proposition~\ref{prop:tail} implies a nontrivial lower bound:
\[
\prob\left( \left\| \frac{p(r,\theta)}{\xi_0} \right\|_{A} < \varepsilon \right) \geq 2 \int_{2ZS/\varepsilon}^{\infty} \left( 1 - \frac{2Z S}{\varepsilon x} \right) e^{-x^2/2} dx /\sqrt{2\pi}.
\]
The factor of 2 arises as the integrand is even.

When the coefficients $\zeta_n$ are complex Gaussians, the value of $\E[|\zeta|]$ is easily determined in closed form, and well-known.
Note that $|\zeta| = \frac{\sqrt{2}}{2} \sqrt{\xi^2 + \eta^2}$ where $\xi$ and $\eta$ are independent standard Gaussians. It follows that
\[
\E[|\zeta|] = \frac{\sqrt{2}}{2} \int_{-\infty}^{\infty} \int_{-\infty}^{\infty} \sqrt{x^2+y^2} e^{-x^2/2} e^{-y^2/2} dx dy / (2\pi).
\]
We can therefore take $Z = \sqrt{\pi}/2$ in the Gaussian case, which completes the proof.

\end{proof}

\section{Proof of Theorem~\ref{thm:main_0}} \label{sec:proof0}

Let $D = B(0,r_0)$ be the disk of radius $r_0$ around 0. We will produce a simply connected nodal domain inside $D$, with non-zero probability.
Recall from (\ref{eq:expected-per-volume}) that using a hexagonal array of several disks packed into a large ball gives
\[
\lim_{{\rm diam}(B) \rightarrow \infty} \frac{\E[N_T(F,B)]}{\vol{B}} \geq \frac{ \prob(N_T(F,D) \geq 1) }{\vol{D} } \frac{\pi}{\sqrt{12}}
\]
where the type $T$ here designates the simply connected nodal domains. 
If $D$ contains any nodal domain, then it must contain a simply connected one as a subdomain of the given domain.
It therefore suffices to estimate the probability that $D$ contains a nodal domain at all.

Let $\delta > 0$ and let $A$ be a domain containing the annulus defined by $|r - j_{0,1} | < \delta$, where $j_{0,1} = 2.4048\ldots$ is the first root of $J_0$.
We assume that $\delta < j_{1,1} - j_{0,1} = 1.42688\ldots$ so that Proposition~\ref{prop:approx} applies to $G = J_0(r)$. Let $x_0$ be any point on the circle $|x|=j_{0,1}$.
Let $F(r,\theta) = J_0(r) + p(r,\theta)/\xi_0$ for $p(r, \theta)$ as in (\ref{e:p_pert}).  Whenever $F$ satisfies
\[
\sup_{B(x_0,\delta)} \max\left( |F- J_0|, |\nabla F - \nabla J_0 | \right) <   \frac{\delta}{1+\delta}  \inf_{B(x_0,\delta)} \left| J_0'(r) \cos{\theta} \right|,
\]
Proposition~\ref{prop:approx} then guarantees that $F$ has a zero inside $B(x_0,\delta)$.  On one side, by (\ref{eq:two-norms}),
\[
\sup_{B(x_0,\delta)} \max( |F - J_0|, |\nabla F - \nabla J_0| ) \leq \sqrt{2} \frac { \| p \|_A  }{ |\xi_0| }
\]
where $A$ is any domain containing $B(x_0,\delta)$. 
On the other,
\[
\inf_{B(x_0,\delta)} | J_0'(r) \cos{\theta} | \geq \left(1- \frac{\delta^2}{j_{0,1}^2} \right)^{1/2} \inf_{|r - j_{0,1}| < \delta} |J_0'(r) |.
\]
Thus $F$ has a zero in $B(x_0,\delta)$ provided that
\begin{equation} \label{eq:howl}
\frac{ \| p \|_{A} }{|\xi_0|} < \frac{1}{\sqrt{2} }  \frac{\delta}{1+\delta}  \left(1 - \frac{\delta^2}{j_{0,1}^2} \right)^{1/2}  \inf_{|r - j_{0,1}| < \delta} |J_0'(r) |.
\end{equation}
This occurs with a non-zero probability quantified by Proposition~\ref{prop:prob}.

The argument so far shows that, whenever $|\xi_0|$ is large enough that (\ref{eq:howl}) holds, $F$ has a nodal line beginning within $\delta$ of the circle $r = j_{0,1}$. 
The next step shows that any such nodal line remains within a narrow annulus around this circle, and in particular does not leave $D$.

\begin{proposition} \label{prop:stay}
Suppose $\delta > 0$ satisfies
\begin{align}
\delta &< j_{1,1} - j_{0,1} = 1.42688\ldots \label{eq:delta-j11-j01}
\end{align}
and $\varepsilon > 0$ satisfies both of the inequalities
\begin{align}
\varepsilon &< |J_0(j_{1,1})| = 0.402759\ldots \label{eq:eps-j11}  \\
\varepsilon &\leq \frac{1}{\sqrt{2} }  \frac{\delta}{1+\delta}  \left( 1 - \frac{\delta^2}{j_{0,1}^2} \right)^{1/2} \inf_{|r - j_{0,1}| < \delta } |J_0'(r) | \label{eq:eps-delt}
\end{align}
Let the interval $[a(\varepsilon), b(\varepsilon)]$ be the connected component of $\{ |J_0(r) | \leq \varepsilon \}$ containing the first root $j_{0,1}$.  Let $A$ be a domain containing both of the annuli $a(\varepsilon) < r < b(\varepsilon)$ and $j_{0,1}-\delta < r < j_{0,1} + \delta$.
Then, whenever $\| p \|_A / |\xi_0| < \varepsilon$, the function $F = J_0(r) + p(r,\theta)/\xi_0$ has a nodal line contained within the annulus $a(\varepsilon) < r < b(\varepsilon)$.
\end{proposition}

We will take $A$ to be the annulus $j_{0,1}-\delta < |x| < j_{0,1}+\delta$.

\begin{proof}[Proof of Proposition~\ref{prop:stay}]
Let $x_0$ be any point on the circle $|x|=j_{0,1}$. 
By (\ref{eq:delta-j11-j01}) and (\ref{eq:eps-delt}), Proposition~\ref{prop:approx} applies to $J_0$ as in Section~\ref{sec:newt}. The ball $B(x_0,\delta)$ is then guaranteed to contain a zero of $F$.
Let $N$ be the nodal line containing one such zero. By the implicit function theorem, $N$ is a smooth curve without boundary.
We claim that $N$ is contained in $A$.

If $F$ vanishes at a point $x \in A$, then
\begin{equation} \label{eq:slab}
|J_0(|x|)| = \frac{ |-p(|x|,\theta)| }{|\xi_0|} \leq \frac{ \| p \|_{A} }{|\xi_0|} < \varepsilon.
\end{equation}
The inequality $|J_0(r)| < \varepsilon$ defines a disjoint union of annuli around the first several zeros of $J_0$, together with one unbounded region since $J_0(r) \rightarrow 0$ as $r \rightarrow \infty$. 
For $\varepsilon \geq 1$, all radii $r > 0$ satisfy $|J_0(r) | < \varepsilon$. For $\varepsilon$ close to 1, the set $\{ |J_0(r)| < \varepsilon \}$ continues to have only one connected component, consisting of all sufficiently large values of $r$. This component splits when $\varepsilon$ crosses a critical value of $J_0$. The first critical point occurs at $j_{1,1} = 3.831705\ldots$, with value $J_0(j_{1,1}) = -0.402759\ldots$, and we assume $\varepsilon < |J_0(j_{1,1})|$ in (\ref{eq:eps-j11}) so that 
$\{| J_0(r)| < \varepsilon \}$ has at least two connected components.
Say the first component of $\{ |J_0(r) | < \varepsilon \}$ is given by $a(\varepsilon) < r < b(\varepsilon)$, and $r > c(\varepsilon)$ for all other solutions.


Notice that $b(\varepsilon) < j_{1,1} < c(\varepsilon)$ since $b(\varepsilon)$ and $c(\varepsilon)$ are on opposite sides of the first critical point $j_{1,1}$. 
Since $j_{0,1} - \delta < |x| < j_{0,1}+\delta$ for all $x \in B(x_0, \delta)$, and $\delta < j_{1,1} - j_{0,1}$ by hypothesis (\ref{eq:delta-j11-j01}), it follows that $|x| < j_{1,1} < c(\varepsilon)$.
This shows that $B(x_0,\delta)$ does not intersect any component of $\{ |J_0(r) | < \varepsilon \}$ except the first annulus $a(\varepsilon) < r < b(\varepsilon)$. 
In particular, $N \cap B(x_0,\delta)$ is contained in $a(\varepsilon) < r < b(\varepsilon)$. 
The nodal line $N$ must remain in the annulus $a(\varepsilon) < r < b(\varepsilon)$ because leaving would make $|J_0(x)| \geq \varepsilon $ for some $x \in N \cap A$, contrary to (\ref{eq:slab}).
As claimed, $N$ is contained in $a(\varepsilon) < r < b(\varepsilon)$ and hence in $A$.

\end{proof}

To complete the proof of Theorem~\ref{thm:main_0}, we take $\delta$ and $\varepsilon$ small enough to satisfy the hypotheses of Proposition~\ref{prop:stay}.  Take $r_0 =  j_{0,1}+\delta$ so that there is a nodal domain contained in the disk $D$ of radius $r_0$.
Then $N_T(F,D)$ is non-zero whenever $\| p \|_A < \varepsilon |\xi_0|$, and
\[
\lim_{{\rm diam}(B) \rightarrow \infty} \frac{\E[N_T(F,B)]}{\vol{B}} \geq \frac{ \prob(N_T(F,D) \geq 1) }{\vol{D} } \frac{\pi}{\sqrt{12}} \geq \frac{1}{\sqrt{12} } \frac{1}{(j_{0,1}+\delta)^2} \prob( \| p \|_A < \varepsilon |\xi_0| ).
\]
The lower bound from Proposition~\ref{prop:prob} then implies
\[
\lim_{{\rm diam}(B) \rightarrow \infty} \frac{\E[N_T(F,B)]}{\vol{B}} \geq \frac{2}{(j_{0,1}+\delta)^2 \sqrt{24 \pi}}  \int_{\sqrt{\pi} S/\varepsilon}^{\infty} \left( 1 - \frac{\sqrt{\pi} S}{\varepsilon x} \right) e^{-x^2/2} dx 
\]
Finally, we divide by $c_{{\rm NS}}$ for normalization:
\[
\mu(0) = \frac{1}{c_{{\rm NS}} } \lim_{{\rm diam}(B) \rightarrow \infty} \frac{\E[N_T(F,B)]}{\vol{B}} \geq \frac{1}{\sqrt{6\pi} } \frac{1}{(j_{0,1}+\delta)^2} \int_{\sqrt{\pi} S/\varepsilon}^{\infty} \left( 1 - \frac{\sqrt{\pi} S}{\varepsilon x} \right) e^{-x^2/2} dx
\]
as claimed.

Numerically, one can take $\delta = 1/2$. 
Then $\varepsilon$ must satisfy
\begin{equation} \label{eq:eps-max}
\varepsilon \leq \frac{1}{\sqrt{2}} \frac{\delta}{1+\delta} \left( 1 - \frac{\delta^2}{j_{0,1}^2} \right)^{1/2} \inf_{|r-j_{0,1}|\leq \delta} |J_0'(r) | = 0.086161\ldots
\end{equation}
the minimum of $|J_0'(r)|$ over $|r - j_{0,1}| \leq \delta$ being achieved by $|J_0'(j_{0,1}+\delta)| = 0.3737\ldots$
The other constraint $\varepsilon < |J_0(j_{1,1})|=0.402759\ldots$ is then satisfied as well.
We choose the largest value of $\varepsilon$ allowed by (\ref{eq:eps-max}).
The resulting endpoints $a_1(\varepsilon), b_1(\varepsilon)$ are then
\begin{align*}
a_1(\varepsilon) &= 2.243784\ldots > j_{0,1}-\delta \\
b_1(\varepsilon) &= 2.577540\ldots < j_{0,1}+\delta
\end{align*}
The domain $A$ is then the annulus $j_{0,1} - \delta < r < j_{0,1} + \delta$. 

Finally, to estimate the probabilities using Proposition~\ref{prop:prob}, we must compute the sum
\[
S = \sum_{n =1}^{\infty} \left(\sup_A |J_n(r)| + \sup_A |J_n'(r) | + n \sup_A \left| \frac{J_n(r) }{r} \right| \right).
\]
Most of the functions $|J_n(r)|$, $|J_n'(r)|$, and $|J_n(r)/r|$ achieve their suprema at the endpoint $r = j_{0,1}+\delta$, except for a handful of special cases when $n=1,2,3$.
First, $|J_1(r)|$ and $|J_1(r)/r|$ achieve their suprema at the other endpoint $j_{0,1}-\delta$.
Second, $|J_2(r)/r|$ achieves its supremum at a critical point inside the interval, and $|J_2'(r)|$ is maximized at $j_{0,1}-\delta$. 
Finally, $|J_3'(r)|$ is maximized at a critical point inside the interval.
The exceptional values, not occurring at either endpoint $j_{0,1} \pm \delta$, are $|J_2(u)/u| = 0.179962\ldots$ and $|J_3'(v)| = 0.187591\ldots$ where $u=2.299910\ldots$ and $v = 2.637911\ldots$ are the respective maxima of $J_2(r)/r$ and $J_3'(r)$.


Summing the separate contributions to $S$ from $n=1,2,3$ and $n\geq 4$ gives
\begin{align*}
 |J_1(j_{0,1}-\delta)| + |J_1'(j_{0,1}+\delta)| + \left| \frac{ J_1(j_{0,1}-\delta)}{j_{0,1}-\delta} \right| &= 1.240843\ldots&\\
  |J_2(j_{0,1}+\delta)| + |J_2'(j_{0,1}-\delta) | + 2 \left| \frac{J_2(u) }{u} \right|&=1.076795\ldots \\
  \left(1 + \frac{3}{j_{0,1}+\delta} \right) |J_3(j_{0,1}+\delta) + |J_3'(v)| &=0.781099\ldots \\
\sum_{n=4}^{\infty} \left( \left( 1 + \frac{n}{j_{0,1}+\delta} \right) |J_n(j_{0,1}+\delta)| + |J_n'(j_{0,1}+\delta)| \right) &=   0.630586\ldots&\\
\end{align*}
and therefore
\begin{equation} \label{eq:s-delta-half}
S = 3.729324\ldots
\end{equation}


With $\varepsilon = 0.086161$ and $S=3.729324$, the bound from Proposition~\ref{prop:prob} is
\[
\prob( \| p \|_A < \varepsilon |\xi_0| ) \geq 2 \int_{\sqrt{\pi}S/\varepsilon}^{\infty} \left(1 - \frac{\sqrt{\pi}S}{\varepsilon x} \right) e^{-x^2/2} dx / \sqrt{2\pi} > 10^{-1280}.
\] To obtain the bound stated in Corollary~\ref{c:numerical_bounds}, we note from (\ref{e:cNS_ub}) that $c_{{\rm NS}} < 10^{-1}$ and therefore
\[
\frac{1}{c_{{\rm NS}} } \frac{1}{(j_{0,1}+\delta)^2 \sqrt{12} } \prob(\| p \|_A < \varepsilon |\xi_0|) \geq 10 \times 0.0342 \times 10^{-1280} > 10^{-1281}.
\]
Further comments on the numerical calculations are given in the appendix.


\section{Proof of Theorem~\ref{thm:main_1} } \label{sec:proof1}

To produce a domain inside another domain, we approximate $J_0(r)$ as before but in a larger region containing its first two roots $j_{0,1} = 2.404825\ldots $ and $j_{0,2} = 5.520078\ldots$ instead of only $j_{0,1}$.
For $\varepsilon > 0$, let $a_1(\varepsilon) < r < b_1(\varepsilon)$ and $ a_2(\varepsilon) < r < b_2(\varepsilon)$ be the first and second connected components of $\{ |J_0(r) | < \varepsilon \}$.
We assume that $\varepsilon$ is less than the second critical value of $J_0$, or else $\{ |J_0(r) | < \varepsilon \}$ would have fewer connected components and $b_2(\varepsilon)$ would effectively be $\infty$. 
Numerically, this requires that
\[
\varepsilon < |J_0(j_{1,2})|=0.300115\ldots
\]
where $j_{1,2}=7.015586\ldots$ is the second positive root of $J_1 = -J_0'$ and hence the second critical point of $J_0$. 
We also write $a_3(\varepsilon)$ for the start of the third component of $\{ |J_0(r)| < \varepsilon \}$, whereas $b_3(\varepsilon)$ might be infinite depending on how $\varepsilon$ compares to the third critical value.

We will argue as before that there are nodal lines $N_1$ and $N_2$ contained in the respective annuli $a_1(\varepsilon) < r < b_1(\varepsilon)$ and $a_2(\varepsilon) < r < b_2(\varepsilon)$, whenever $\varepsilon$ is small enough and $\| p \|_A / |\xi_0| < \varepsilon$ for a suitable domain $A$. Unlike the previous case, a further step is needed here to guarantee that $N_2$ surrounds $N_1$ and yields a nodal domain of connectivity 1.
One way to do so, without needing stronger restrictions on the perturbation $p$, uses the fact that nodal domains have a non-trivial minimal volume. 

\begin{proposition} \label{prop:minvol}
If $(\Delta + 1)F = 0$, then any nodal domain of $F$ has volume at least $\pi j_{0,1}^2 = 18.168414\ldots$ where $j_{0,1}$ is the first root of the Bessel function $J_0$.
\end{proposition}
\begin{proof}[Proof of Proposition~\ref{prop:minvol}]
This is a standard consequence of the Faber-Krahn inequality, which asserts that the ball has minimal $\lambda_1(U)$ among all domains $U$ of equal volume, where $\lambda_1(U)$ denotes the lowest eigenvalue of $\Delta$ with Dirichlet boundary conditions on $\partial U$. 
Let $U$ be a nodal domain of $F$. We have $\lambda_1(U) \leq 1$ because $(\Delta + 1)F = 0$ and, by definition, $F$ vanishes on $\partial U$.
The ball of equal volume has radius $\sqrt{\vol(U)/\pi}$ and first eigenfunction $J_0(r j_{0,1} \sqrt{\pi / \vol(U) } )$, where the scaling by $j_{0,1} \sqrt{\pi / \vol(U) }$ guarantees the boundary conditions.
The resulting eigenvalue is then $j_{0,1}^2 \pi / \vol(U)$. 
By Faber-Krahn, $1 \geq \lambda_1 \geq j_{0,1}^2 \pi / \vol(U)$, as required.
\end{proof}

In particular, if $\varepsilon$ is small enough, then the annuli $a_k(\varepsilon) < r < b_k(\varepsilon)$ for $k=1,2$ have volume too small to contain a complete nodal domain. This will force $N_1$ and $N_2$ to bound a domain of connectivity exactly 1.
The correct connectivity could also be guaranteed without the minimal volume property, by taking $F$ even closer to $J_0$ in the $C^1$ topology if necessary.

\begin{proposition} \label{prop:stay1}
Suppose $\delta > 0$ satisfies
\begin{align}
\delta &< j_{1,1} - j_{0,1} = 1.42688\ldots \label{eq:delt}
\end{align}
and $\varepsilon > 0$ satisfies
\begin{align}
\varepsilon &< |J_0(j_{1,2})| = 0.300115\ldots  \label{eq:eps-crit}\\
\varepsilon &\leq \frac{1}{\sqrt{2} } \frac{\delta}{1+\delta} \left( 1 - \frac{\delta^2}{j_{0,1}^2} \right)^{1/2} \inf_{|r - j_{0,1}| < \delta } |J_0'(r) | \label{eq:eps-b1} \\
\varepsilon &\leq \frac{1}{\sqrt{2} } \frac{\delta}{1+\delta}  \left( 1 - \frac{\delta^2}{j_{0,2}^2} \right)^{1/2} \inf_{|r - j_{0,2}| < \delta } |J_0'(r) | \label{eq:eps-b2}
\end{align}
Let the intervals $[a_1(\varepsilon),b_1(\varepsilon)]$, $[a_2(\varepsilon),b_2(\varepsilon)]$, and $[a_3(\varepsilon),b_3(\varepsilon)]$ be the connected components of $\{ |J_0(r) | \leq \varepsilon \}$ containing, respectively, the first, second, and third roots $j_{0,1}$, $j_{0,2}$, $j_{0,3}$ of $J_0$.
Assume further that
\begin{align}
&\delta < \min\big( a_2(\varepsilon) - j_{0,1}, j_{0,2} - b_1(\varepsilon), a_3(\varepsilon) - j_{0,2} \big) \label{eq:delta-diffs1} \\
&b_1(\varepsilon)^2 - a_1(\varepsilon)^2 < j_{0,1}^2 \label{eq:minvol1} \\
&b_2(\varepsilon)^2 - a_2(\varepsilon)^2 < j_{0,1}^2 \label{eq:minvol2}
\end{align}
Let $A$ be a domain containing all of the annuli
\[
\begin{aligned}
&\{ a_1(\varepsilon) <r<b_1(\varepsilon) \}& && && & \{ j_{0,1}-\delta <r<j_{0,1}+\delta \}& \\
& \{a_2(\varepsilon) <r<b_2(\varepsilon)\} & && && & \{ j_{0,2}-\delta <r<j_{0,2}+\delta \}&
\end{aligned}
\]
Then, whenever $\| p \|_A / |\xi_0| < \varepsilon$, the function $F = J_0(r) + p(r,\theta)/\xi_0$ has a nodal domain of connectivity 1, with one boundary component contained in the annulus $a_1(\varepsilon) < r < b_1(\varepsilon)$ and the other contained in $a_2(\varepsilon) < r < b_2(\varepsilon)$. 
\end{proposition}

In practice, $A$ will be the annulus $j_{0,1} - \delta < r < j_{0,2}+\delta$. The values $a_k(\varepsilon)$ and $b_k(\varepsilon)$ can be determined by numerically solving the equation $J_0(r) = \pm \varepsilon$ over different ranges of $r$.

\begin{figure}
\centering
\begin{tikzpicture}[scale=0.95]
\definecolor{1}{gray}{0.9};
\definecolor{2}{gray}{0.8};
\filldraw[fill=1,draw=1] (0,0) circle [radius=3];
\filldraw[fill=white,draw=white] (0,0) circle [radius=2.5];
\filldraw[fill=1,draw=1] (0,0) circle [radius=1.5];
\filldraw[fill=white,draw=white] (0,0) circle [radius=1];
\filldraw[fill=2,draw=2] (1,0) circle (0.6);
\filldraw[fill=2,draw=2] (3,0) circle (0.6);
\draw (0,0) circle (1.25);
\draw (0,0) circle (2.75);
\draw[blue] [samples=100,domain=0:360] plot ({1.25*cos(\x)+0.1*cos(10*\x)},{1.25*sin(\x)});
\draw[blue] [samples=100,domain=0:360] plot ({-0.1+2.75*cos(\x)+0.1*cos(12*\x)},{2.75*sin(\x)});
\draw[blue] (0,1.5) node {$N_1$};
\draw[blue] (0,3) node {$N_2$};
\filldraw[fill=1,draw=1] (0+7.5,0) circle [radius=3];
\filldraw[fill=white,draw=white] (0+7.5,0) circle [radius=2.5];
\filldraw[fill=1,draw=1] (0+7.5,0) circle [radius=1.5];
\filldraw[fill=white,draw=white] (0+7.5,0) circle [radius=1];
\draw (0+7.5,0) circle (1.25);
\draw (0+7.5,0) circle (2.75);
\draw[red] (1+7.5,0) circle (0.618 and 1);
\draw[red] (2.875+7.5,0)  arc[radius = 2.875, start angle= 0, end angle= 270];
\draw[red] (2.625+7.5,0)  arc[radius = 2.625, start angle= 0, end angle= 270];
\draw[red] (0+7.5,-2.875) .. controls (0+7.5,-2.875) and (0.5+7.5,-2.875) .. (0.5+7.5,-2.75) .. controls (0.5+7.5,-2.625) and (0.5+7.5,-2.625) .. (0+7.5,-2.625);
\draw[red] (2.875+7.5,0) .. controls (2.875+7.5,0) and (2.875+7.5,-0.5) .. (2.75+7.5,-0.5) .. controls (2.625+7.5,-0.5) and (2.625+7.5,-0.5) .. (2.625+7.5,0);
\end{tikzpicture}
\caption{
At left: nodal lines $N_1$, $N_2$ as in Proposition~\ref{prop:stay1}, drawn in blue. The circles on which $J_0(r)$ vanishes are drawn in black. The shaded annuli are the first two components of the region $\{ |J_0(r)| > \varepsilon \}$, bounded by $a_k(\varepsilon)$ and $b_k(\varepsilon)$. The balls $B(x_1,\delta)$ and $B(x_2,\delta)$ are shaded in dark gray. At right: an inner nodal line that cannot occur because it would enter the region where $|J_0(r)| > \varepsilon$, and an outer nodal line that cannot occur because it would bound a nodal domain with volume less than the minimum allowed by Faber-Krahn. The same two reasons rule out any additional nodal lines from increasing the connectivity of the domain bounded by $N_1$ and $N_2$.
For illustrative purposes, the figure is not drawn to scale. 
}
\end{figure}
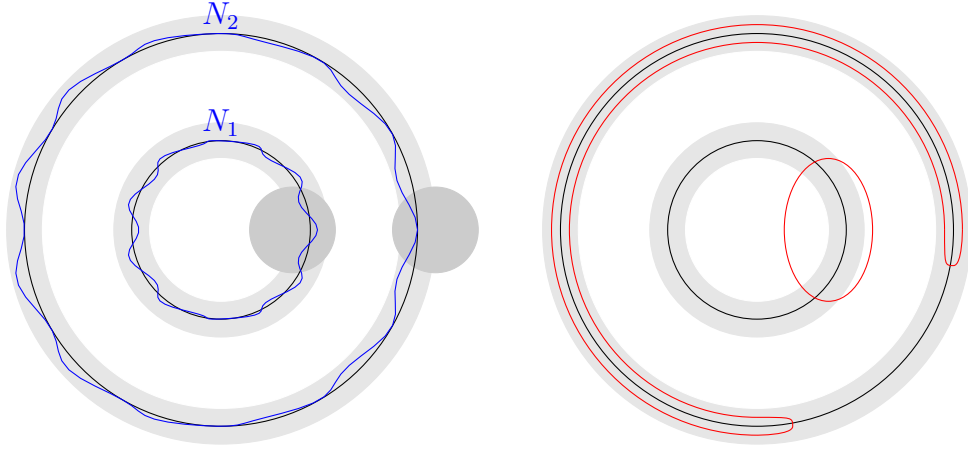

\begin{proof}[Proof of Proposition~\ref{prop:stay1}]
Let $x_1 = (j_{0,1},0)$ and $x_2 = (j_{0,2},0)$ be points on the circles $r=j_{0,1}$ and $r=j_{0,2}$.
We will show first that, with non-zero probability, $F$ has a zero in each ball $B(x_1,\delta)$ and $B(x_2,\delta)$.
Let $A$ be a domain containing both $B(x_1,\delta)$ and $B(x_2,\delta)$.
In the event that $\| p \|_A < \varepsilon |\xi_0|$, which occurs with non-zero probability, Proposition~\ref{prop:approx} applies to $F=J_0(r) + p(r,\theta)/\xi_0$.
The result is that, if
\[
\varepsilon < \frac{\delta}{1+\delta} \inf_{ B(x_1,\delta) \cup B(x_2,\delta) } |J_0'(r) \cos{\theta} |
\]
then $F$ has a zero in each ball $B(x_1,\delta)$ and $B(x_2,\delta)$. 
To obtain a non-trivial result, we must assume $\delta$ is small enough that the infimum is non-zero. We claim that (\ref{eq:delt}) suffices.

As in the proof of Proposition~\ref{prop:stay}, (\ref{eq:delt}) is enough to apply Proposition~\ref{prop:approx} to $J_0$ in $B(x_1,\delta)$.  We claim that this assumption is also enough to apply Proposition~\ref{prop:approx} in $B(x_2,\delta)$.  To produce a zero near $x_2$, we need $J_0'(r) \cos{\theta} \neq 0$ on $B(x_2,\delta)$.
This will follow provided that $J_0'(r) \neq 0$ for $j_{0,2} - \delta < r < j_{0,2} + \delta$ and $\cos{\theta} \neq 0$ for the range of angles occurring for points in $B(x_2,\delta)$. 
Since $J_0' = -J_1$, the radial requirement is that the interval $[j_{0,2}-\delta, j_{0,2}+\delta]$ must not contain any zeros of $J_1$.
This constrains $\delta$ further:
\[
j_{1,1} < j_{0,2} - \delta \qquad \text{and} \qquad j_{0,2} + \delta < j_{1,2}
\]
where $j_{1,k}$ denotes the $k$-th positive root of $J_1$.
These constraints amount to $\delta < j_{0,2} - j_{1,1} = 1.6883\ldots$ and $\delta < j_{1,2} - j_{0,2}=1.4955\ldots$ , both of which are implied by (\ref{eq:delt}).
In the angular variable, we have a lower bound
\[
\cos{\theta} > \sqrt{ 1 - \frac{\delta^2}{j_{0,2}^2 } }
\]
so that $\cos{\theta} \neq 0$ within $B(x_2,\delta)$ as long as $\delta < j_{0,2} = 5.52\ldots$, which certainly follows from the constraints imposed earlier.  As claimed, (\ref{eq:delt}) by itself suffices to produce a zero near each of $x_1$ and $x_2$.

To guarantee a separate zero in each ball, we assume first of all that $B(x_1,\delta)$ and $B(x_2,\delta)$ are disjoint. This requires $j_{0,1} + \delta < j_{0,2} - \delta$, or
\[
\delta < \frac{j_{0,2} - j_{0,1} }{2} = 1.557626\ldots
\]
which follows once again from (\ref{eq:delt}).
We must also prevent $N_1$ and $N_2$ from merging outside of their starting balls, and this is the reason for (\ref{eq:delta-diffs1}). To do so, we will choose $\varepsilon$ small enough that each $N_k$ is contained in its respective annulus $a_k(\varepsilon) < r < b_k(\varepsilon)$.
The argument is as before, except the domain $A$ must be large enough to include both balls $B(x_k,\delta)$ and both annuli $\{ a_k(\varepsilon) < r < b_k(\varepsilon) \}$.  The assumption (\ref{eq:delta-diffs1}) guarantees that the intersection of $B(x_k,\delta)$ with $\{ |J_0| \leq \varepsilon \}$ occurs in a single annulus $\{ a_k(\varepsilon) < r < b_k(\varepsilon) \}$ rather than straddling multiple components.  Indeed, since $j_{0,1}+\delta < a_2(\varepsilon)$, $B(x_1,\delta)$ cannot intersect $a_2(\varepsilon) < r < b_2(\varepsilon)$ or any later components. 
Likewise, $j_{0,2} - \delta > b_1(\varepsilon)$ and $j_{0,2}+\delta < a_3(\varepsilon)$ ensure that $a_2(\varepsilon) < r < b_2(\varepsilon)$ is the only one intersected by $B(x_2,\delta)$.
This confirms that $N_1$ intersects $a_1(\varepsilon) < r < b(\varepsilon)$, while $N_2$ intersects $a_2(\varepsilon) < r < b_2(\varepsilon)$. 
Then (\ref{eq:slab}) prevents the nodal lines from leaving the annuli in which they begin.  In particular, $N_1$ and $N_2$ are distinct.

Finally, the role of (\ref{eq:minvol1}) and (\ref{eq:minvol2}) is to guarantee that each of the annuli $a_1(\varepsilon) < r < b_1(\varepsilon)$ and $a_2(\varepsilon) < r < b_2(\varepsilon)$ has volume less than $\pi j_{0,1}^2$. By Proposition~\ref{prop:minvol}, it impossible for either region to fully contain a nodal domain. Together with (\ref{eq:slab}), which prevents all nodal lines from intersecting $\{ |J_0|> \varepsilon \} \cap A$, this rules out any extraneous nodal lines aside from $N_1$ and $N_2$ in the annulus $A$.
The minimal volume property also shows that each of $N_1$ and $N_2$ has winding number 1 around the origin, or else they would enclose an impossibly small volume. 
Lastly, these arguments eliminate the possibility of nodal lines contained entirely within $B(j_{0,2}+\delta) \setminus A$, since a ball of radius $j_{0,1}-\delta$ or $a_1(\varepsilon)$ has volume smaller than allowed by Proposition~\ref{prop:minvol}.
As claimed, $N_1$ and $N_2$ are the sole boundary components of a domain of connectivity exactly 1.  

\end{proof}

To complete the proof of Theorem~\ref{thm:main_1}, we take $\delta$ and $\varepsilon$ small enough to satisfy the hypotheses of Proposition~\ref{prop:stay1} and let $A$ be a sufficiently large annulus as above.  Take $R =  j_{0,2}+\delta$ so that there is a nodal domain of connectivity 1 contained in the disk $D$ of radius $R$, whenever $\| p \|_{A} < \varepsilon |\xi_0|$. The rest of the proof is the same as for Theorem~\ref{thm:main_0}, with $j_{0,2}$ in place of $j_{0,1}$. 

Let us take $\delta=1/2$ as before. The largest $\varepsilon$ allowed by Proposition~\ref{prop:stay1} is
\[
\varepsilon = \frac{1}{\sqrt{2} } \frac{\delta}{1+\delta}  \left( 1 - \frac{\delta^2}{j_{0,2}^2} \right)^{1/2} \inf_{|r - j_{0,2}| < \delta } |J_0'(r) | = 0.064008\ldots
\]
For this choice,
\[
\begin{aligned}
&a_1(\varepsilon)=2.284353\ldots & && &b_1(\varepsilon)=2.531685\ldots & \\
&a_2(\varepsilon)=5.334081\ldots & && &b_2(\varepsilon)=5.712642\ldots & \\
&a_3(\varepsilon)=8.418990\ldots & && && \\
\end{aligned}
\]
In particular, the constraints (\ref{eq:delta-diffs1}), (\ref{eq:minvol1}), and (\ref{eq:minvol2}) are all satisfied.
We take the domain $A$ to be the annulus $j_{0,1}-\delta < r < j_{0,2}+\delta$. 
For all but finitely many $n$, the functions $|J_n(r)|$, $|J_n'(r)|$, and $|J_n(r)/r|$ achieve their maxima at the upper endpoint $r=j_{0,2}+\delta$. There are exceptions for $1 \leq n \leq 6$. 
We summarize their contributions to $S$ in Table~\ref{table:contributions}. For $n \geq 7$, all the maxima are achieved at $j_{0,2}+\delta$ and the sum is
\[
S_{\geq 7} = \sum_{n=7}^{\infty} \left( \left(1 + \frac{n}{j_{0,2}+\delta} \right)|J_n(j_{0,2}+\delta)| + |J_n'(j_{0,2}+\delta)|  \right) = 0.689769\ldots
\]
The final tally is
\[
S = S_1 + S_2+S_3+S_4+S_5 + S_6 + S_{\geq 7} = 5.215701\ldots
\]
With $\varepsilon = 0.064008\ldots$ and $S=5.215701\ldots$, the bound from Proposition~\ref{prop:prob} is
\[
\prob( \| p \|_A < \varepsilon |\xi_0| ) \geq 2 \int_{\sqrt{\pi}S/\varepsilon}^{\infty} \left(1 - \frac{\sqrt{\pi}S}{\varepsilon x} \right) e^{-x^2/2} dx / \sqrt{2\pi} > 10^{-4532}.
\]

\begin{table} \label{table:contributions}
\caption{
Each $n$ contributes $S_n = |J_n(u_n)|+n|J_n(v_n)/v_n| + |J_n'(w_n)|$ to the final sum $S$, where $u_n$, $v_n$, $w_n$ are where $|J_n(r)|$, $|J_n(r)/r|$, and $|J_n'(r)|$ achieve their maxima over the interval $j_{0,1}-\delta \leq r \leq j_{0,2}+\delta$ with $\delta=1/2$. 
The maxima are achieved at critical points, except $u_1 = j_{0,1}-\delta$ and $w_3=u_5=u_6=v_6=j_{0,2}+\delta$. Values have been truncated to 4 digits.}
\begin{tabular}{rlllllll}
$n$ & $u_n$ & $|J_n(u_n)|$  & $v_n$  & $|J_n(v_n)/v_n|$ & $w_n$ & $|J_n'(w_n)|$ & $S_n$\\
\midrule
$1$ & 1.9048 & 0.5810 & $5.1356$ & 0.0661 & $3.5183$ & 0.4194 & 1.0666\\
$2$ & $3.0542$ & 0.4864 & $2.2999$ & 0.1799 & $4.8879$ & 0.3478 & 1.0143\\
$3$ & $4.2011$ & 0.4343 & $3.6112$ & 0.1107 & $6.0200$ & 0.3009 & 0.8461\\
$4$ & $5.3175$ & 0.3996 & $4.8112$ & 0.0787 & $3.6804$ & 0.1548 & 0.6333\\
$5$ & $6.0200$ & 0.3631 & $5.9623$ & 0.0603 & $4.7082$ & 0.1338 & 0.5573\\
$6$ & $6.0200$ & 0.2481 & $6.0200$ & 0.0412 & $5.7285$ & 0.1188 & 0.4082\\
\end{tabular}
\end{table}

\section{ Proof of Theorem~\ref{thm:sym} } \label{sec:symmetrization}

The radial symmetrization of $F: \R^2 \rightarrow \R$ around a point $z$ is defined as
\begin{equation} \label{e:sz}
S_z F(x) = \frac{1}{2\pi} \int_{S^1} F(\theta x) d\theta
\end{equation}
where $\theta x$ denotes the rotation of $x$ by angle $\theta$ around $z$, and $d\theta$ denotes the infinitesimal arc length on the circle $S^1$. 
By construction, $S_z f(x)$ depends only on $|x-z|$.
If $F(z)$ solves $\Delta F + F = 0$, then so does its symmetrization. Up to a constant multiple, there is only one radial eigenfunction around $z$ with no singularity at $z$, and it follows that
\begin{equation} \label{e:eigen-sym}
S_z F(x) = F(z) J_0(|x-z|).
\end{equation}
Moreover, if $F$ has no zeros on some circle around $z$, then $S_z F$ and $F$ have the same sign on that circle.
This leads to the following proposition, which can be used to show that $F$ and $J_0$ have similar nodal sets. 
The expected number of intersections between $F^{-1}(0)$ and a given curve can be calculated using the Kac-Rice formula, and this gives some control on the probability with which $F$ will have no zeros on a circle.

\begin{proposition} \label{prop:sym}
Suppose that $\Delta F + F = 0$ and $F$ has no zeros on a circle $|x-z| = r$ of radius in the interval
\[
j_{0,k-1} < r < j_{0,k}
\]
where $0 < j_{0,1} < \ldots < j_{0,k} < \ldots$ are the roots of the Bessel function $J_0$, and for uniformity of notation, we put $j_{0,0}=0$.
Then $F$ and its symmetrization $S_z F$ have the same sign on that circle.
This sign is opposite to that of $F(z)$ if $k$ is even, or equal to the sign of $F(z)$ if $k$ is odd.
\end{proposition}

\begin{proof}
Suppose $F$ has no zeros on a circle $|x-z|= r$. Then $F$ and $S_z F$ have the same sign on that circle. 
In (\ref{e:eigen-sym}), the factor $J_0(|x-z|)$ is positive for $|x-z| < j_{0,1}$, negative on intervals of the form $j_{0,2m-1} < |x-z| < j_{0,2m}$, and positive again for $j_{0,2m} < |x-z| < j_{2m+1}$.
Thus (\ref{e:eigen-sym}) shows that the sign of $S_z F(x)$ is equal or opposite to the sign of $F(z)$ as claimed.
\end{proof}

We write $G(r)$ for the set of centers $z \in \R^2$ such that $F$ has no zeros on the circle of radius $r$ around $z$. Suppose $0 \in G(r_1) \cap \ldots \cap G(r_M)$, where $M \geq 1$ and the radii satisfy
\begin{equation} \label{e:interlace}
j_{0,1} < r_1 < j_{0,2} <  \ldots < j_{0,M} < r_M < j_{0,M+1}.
\end{equation}
Assume also that $F(0) \neq 0$; this will almost surely be the case when $F$ is a monochromatic random wave. 
The locations of $r_k$ among the Bessel zeros guarantee, by Proposition~\ref{prop:sym}, that $F$ has alternating signs on the circles of radius $r_1,\ldots,r_M$ around 0, with the sign on $|x|=r_1$ opposite to that of $F(0)$. 
It follows that $F$ has a nodal line $N_1$ contained in the ball $|x| < r_1$, and in case $M \geq 2$, another nodal line $N_k$ in each annulus $r_{k-1} < |x| < r_{k}$ for $k=2,\ldots,M$. It is convenient to write $A_k$ for the annulus $r_{k-1} < |x| < r_k$, and introduce another radius $r_0 = 0$ so that $N_k \subset A_k$ for all $k$, including $k=1$.

We would like to conclude that the ball $B(r_M)$ around 0 contains a chain of $M$ nodal domains, each nested in the next and with no other domains branching off from this chain. 
This would yield a simply connected domain if $M=1$, or a domain of connectivity 1 if $M=2$, or further nesting for larger values of $M$.
The obstacle is that $N_1,\ldots, N_M$ need not be the only nodal lines contained in $B(r_M)$. 
In some cases, extra nodal lines can simply be included as part of the chain, and one finds the desired configuration in a ball even smaller than $B(r_M)$. For example, if the goal is to produce a simply connected domain, then an arbitrary nesting configuration suffices because one can always pass to one of the configuration's innermost domain(s).
For $M \geq 2$, it is necessary to rule out some configurations, and this can be done using the sign conditions on $F$ and the fact that each nodal domain occupies a certain volume.

%

As a consequence of Proposition~\ref{prop:minvol}, we have
\begin{proposition} \label{prop:max-per-annulus}
The number of nodal domains contained fully within the annulus $A_k$ is at most
\[
\frac{r_k^2 - r_{k-1}^2 }{j_{0,1}^2}.
\]
In particular, if $r_1 < \sqrt{2} j_{0,1} = 3.4009\ldots$, then there can be at most one nodal domain contained entirely within $A_1$, and if $r_k^2 - r_{k-1}^2 < j_{0,1}^2$, then there cannot be any nodal domains contained entirely within $A_k$.
\end{proposition}

These volume considerations are enough to rule out extra nodal lines and show that $F$ has a nodal domain of connectivity exactly 1.
\begin{lemma} \label{lem:con1}
Suppose $r_1 < \ldots < r_M$ satisfy
\begin{align}
j_{0,1} &< r_1 < \sqrt{2} j_{0,1}  \label{e:r1} \\
j_{0,2} &< r_M < j_{0,3} \label{e:rm}
\end{align}
and, for $2 \leq k \leq M$,
\begin{equation}
r_k^2 - r_{k-1}^2 < j_{0,1}^2. \label{e:steps}
\end{equation}
Assume $\Delta F + F = 0$ and that $F$ has no zeros on the circles $|x|=r_k$.
Suppose also that $F(0) \neq 0$ and the zero set of $F$ is regular.
Then $F$ has a nodal domain of connectivity 1 contained in the ball $B(r_M)$ around 0.
\end{lemma}

\begin{proof}
Write $A_k$ for the annulus $r_{k-1} < |x| < r_k$, and $A_1$ for the ball $|x| < r_1$.
Each nodal line of $F$ intersecting the ball $B(r_M)$ must lie within one of the annuli $A_k$ because $F$ has no zeros on any of the circles $|x| = r_k$ separating them.
Proposition~\ref{prop:sym} implies that $F$ has a nodal line $N_1$ contained in $A_1$, the range (\ref{e:r1}) being sufficient because $\sqrt{2}j_{0,1} < j_{0,2}$.
Likewise, by (\ref{e:rm}), the annulus between $r_1$ and $r_M$ contains another nodal line of $F$, which is disjoint from $N_1$ because $F$ has no zeros on the circle $|x| = r_1$. 

By Proposition~\ref{prop:max-per-annulus}, (\ref{e:r1}) and (\ref{e:steps}) imply that there is at most one nodal domain contained entirely within the ball $A_1$, and none within any annulus $A_k$ for $k \geq 2$. 
In particular, $N_1$ is the only nodal line contained in $B(r_1)$.

Suppose $A_k$ is the first annulus for $k \geq 2$ to contain a nodal line of $F$, say $N_2$. Then $N_2$ must surround $N_1$, and there can be no other nodal lines in $A_k$, or else there would be a nodal domain contained entirely within $A_k$, contrary to (\ref{e:steps}). It follows that $N_1$ and $N_2$ bound a nodal domain of connectivity 1.

\end{proof}

The next step quantifies the probability of satisfying the hypothesis $0 \in G(r_1) \cap \ldots \cap G(r_M)$.
The case of a single radius was proved in \cite{IR19}.
\begin{proposition} \label{prop:kr}
Let $0 < r_1 < \ldots < r_M$ with $J_0(r_k) \neq 0$.
For any $T > 0$, the random wave $\xi_0 J_0(r) + p(r,\theta)$ as in (\ref{eq:plane_wave_complex}) satisfies
\[
\prob \left( 0 \in \bigcap_k G(r_k) \right) \geq 2 \int_T^{\infty} \left( 1 - \frac{\sqrt{2}}{2} \sum_k  \frac{r_k}{\sqrt{1-J_0(r_k)^2} } \exp\left( \frac{-t^2 J_0(r_k)^2 }{2(1-J_0(r_k)^2)}\right) \right) d\gamma(t)
\]
where $d\gamma(t) = e^{-t^2/2} dt / \sqrt{2\pi}$ is the density of a standard Gaussian.
\end{proposition}
\begin{proof}
By a union bound,
\begin{align*}
\prob(0 \in G(r_1) \cap \ldots \cap G(r_M) ) &= 1 - \prob( 0 \notin G(r_1) \cup \ldots \cup G_{r_M} ) \\
&\geq 1 - \sum_k \prob(0 \notin G(r_k) )
\end{align*}
Let $Z=Z(r,\xi_0)$ be the number of zeros of $F$ on the circle $|x|=r$, conditioned on $\xi_0$. 
One has $0 \in G(r)$ if and only if $Z=0$.
As above, write $F(x) = \xi_0 J_0(r) + p(r,\theta)$ in polar coordinates.
Fix $r$ and think of $Z = Z(\xi_0)$ as the number of angles $\theta \in S^1$ such that
\[
p(r,\theta) = - \xi_0 J_0(r)
\]
Then
\[
\prob(0 \in G(r)) = \int_{\R} \prob( Z(\xi_0) = 0) d\gamma(\xi_0)
\]
where $d\gamma(\xi_0) = \exp(-\xi_0^2/2) d\xi_0 / \sqrt{2\pi}$ is the density of a standard Gaussian $\xi_0$. 

Almost surely, each intersection of $F^{-1}(0)$ with the circle $|x| = r$ is transverse. It follows that $Z$ is almost surely even, and hence at least 2 if non-zero. Therefore
\[
\prob( 0 \in G(r)) = \int \left( 1 - \prob(Z(\xi_0) \geq 2) \right) d\gamma(\xi_0).
\]
For $|\xi_0| \geq T$, we use Markov's inequality
\[
\prob( Z(\xi_0) \geq 2) \leq \frac{1}{2} \E[Z(\xi_0)]
\]
where the expectation is taken over $p(r,\theta)$ conditioned on $\xi_0$.
Following \cite[Lemma 2.1]{IR19}, this expectation can be computed using the Kac-Rice formula.
The result is
\[
\E[Z(\xi_0)] = \sqrt{2} \frac{r}{\sqrt{1-J_0(r)^2} } \exp\left( \frac{-\xi_0^2 J_0(r)^2 }{2(1-J_0(r)^2)}\right).
\]
Using this estimate for $|\xi_0| > T$ and the trivial bound $\prob(\ldots) \geq 0$ for $|\xi_0| < T$, we arrive at
\[
\prob( 0 \in G(r)) \geq 2 \int_T^{\infty} \left( 1 - \frac{\sqrt{2}}{2} \frac{r}{\sqrt{1-J_0(r)^2} } \exp \left( \frac{-\xi_0^2 J_0(r)^2 }{2(1-J_0(r)^2)}\right) \right) d\gamma(\xi_0).
\]
Taking $r=r_k$ successively, we obtain, for any choice of $T > 0$,
\[
\prob \left( 0 \in \bigcap_k G(r_k) \right) \geq 2 \int_T^{\infty} \left( 1 - \frac{\sqrt{2}}{2} \sum_k  \frac{r_k}{\sqrt{1-J_0(r_k)^2} } \exp\left( \frac{-\xi_0^2 J_0(r_k)^2 }{2(1-J_0(r_k)^2)}\right) \right) d\gamma(\xi_0)
\]
as required, by the union bound.
\end{proof}

The optimal choice of $T$ satisfies
\begin{equation} \label{e:T-optimal}
1 - \frac{\sqrt{2}}{2} \sum_k  \frac{r_k}{\sqrt{1-J_0(r_k)^2} } \exp\left( \frac{-T^2 J_0(r_k)^2 }{2(1-J_0(r_k)^2)}\right) = 0.
\end{equation}
This is the smallest value that makes the integrand positive for $|\xi_0| > T$.
If there is only a single radius $r$, one finds a solution in closed form:
\[
T^2 = \left( \frac{1}{J_0(r)^2} - 1 \right) \log\left( \frac{r^2}{2(1-J_0(r)^2)} \right).
\]

Numerically, we take a limiting case
\[
r_k \rightarrow \sqrt{k+1} j_{0,1}
\]
for $k= 1,2,3,4,5$, so that the volume constraints (\ref{e:r1}) and (\ref{e:steps}) just barely hold. Note that $r_5 = \sqrt{6} j_{0,1} = 5.8905$ lies between $j_{0,2}=5.5200$ and $j_{0,3}=8.6537$, so one can take $M=5$ in (\ref{e:rm}). 
This choice gives
\[
\mu(1) > 3.2724 \times 10^{-247}
\]
after solving for $T = 41.9286$ using (\ref{e:T-optimal}).
For $\mu(0)$ on the other hand, one can take $M=1$ and a single radius $r_1 = 3.8317$ equal to the critical point of $J_0$ between $j_{0,1}$ and $j_{0,2}$. This choice makes $J_0(r)$ as small as possible, which does not quite optimize the bound but comes close. The result is
\[
\mu(0) > 2.1186 \times 10^{-5}
\]
with $T = 3.2086$. The appendix discusses the numerical details.

\section{Conclusion} \label{sec:conc}

We have given the first effective bounds on the monochromatic connectivity measure in two dimensions.
The arguments are simple and robust, and would be easy to adapt to other situations.
We have relied only in inessential ways on special features such as the differential equation satisfied by the random functions, or the fact that the random coefficients are Gaussian. 
It would be enough to have an explicit basis for the model, such as (\ref{e:plane_wave}), and a bound on the first moment of the coefficients.
The only fact specific to eigenfunctions that enters in the proofs of Theorems~\ref{thm:main_0} and \ref{thm:main_1} is that their nodal domains have a certain minimal volume, but this can be avoided by further refining the $C^1$ approximation.

Numerical efficiency was not our goal at this stage.
We highlight two ways this aspect could be improved, one deterministic and one probabilistic.
A third improvement is of course to optimize the choice of $\delta$ and $\varepsilon$, but the following variations seem likely to be more impactful.
The first observation is that there are many eigenfunctions, not particularly close to our target $J_0(|x|)$, but which nevertheless have a similar nodal topology.
Suppose $G_j$ are finitely many Laplace eigenfunctions, each with a nodal domain of the desired type contained in a domain $A$. 
Suppose one has determined numbers $\varepsilon_j>0$ small enough that an approximation of the form $\| F - G_j \|_{A} < \varepsilon_j$ guarantees that $F$ then also has such a nodal domain. 
By taking smaller parameters if necessary, one can assume $\| G_j - G_k \|_{A} > 2 \max_j \varepsilon_j$ for all $j \neq k$.
Then the events $\| F - G_j \|_{A} < \varepsilon_j$ are disjoint and the probability that $F$ has a nodal domain is at least
\[
\sum_j \prob( \| F - G_j \| < \varepsilon_j ).
\]
This sum could be much higher than the result from using only a single function $G$, provided one can supply enough well-separated targets.

For a concrete example of how to produce targets other than $J_0(|x|)$, consider a sum $\sum_j c_j J_0(|x-z_j|)$. If all the centers $z_j$ coincide, say $z_j=0$, then the sum has the same zero set as $J_0(|x|)$. As $z_j$ and $c_j$ vary, one can produce a simply connected domain whose shape may be quite different from the round disk captured by approximations to $J_0(|x|)$. 
For example, taking $z_1$ and $z_2$ reasonably close together leads to a figure-eight or peanut-shaped domain.
With a larger number of centers, one could produce nodal domains of arbitrarily high complexity.
In particular, this is one approach to quantitative estimates of $\mu(h)$ for higher $h$.

The second improvement is that, even with a single target $J_0$, using stronger probabilistic input would no doubt strengthen the lower bounds stated in Corollary~\ref{c:numerical_bounds}, perhaps to something closer to $10^{-100}$ or $10^{-10}$, rather than $10^{-1000}$.
Proposition~\ref{prop:prob} uses nothing more than Markov's inequality, which is easy to implement and would allow one to study other models with very general random coefficients, but falls far short of capturing the tighter concentration present in the Gaussian case. Much more powerful bounds on suprema are available, in particular from Dudley's entropy method \cite{D67}. 
See \cite{BL13} and \cite{CHsup} for bounds on suprema in the random wave model.

The method could also be improved at the level of detecting individual nodal domains.
We have simply packed equal-sized disks and argued that each has a non-zero probability of containing a nodal domain. 
One can replace these disks with detectors of different shapes and sizes, and these choices can be made non-deterministically, that is, in response to the random outcome $F$.
This extra flexibility could become more important as one considers more elaborate nodal topologies, where surrounding each instance by a large ball might interfere with detecting other configurations.

\section*{Acknowledgements}

It is a pleasure for us to thank Alex Barnett and Alejandro Rivera for helpful discussions about nodal domains of random waves.  SE is partially supported by NSERC Discovery Grant RGPIN-2020-04775.  SE would also like to thank the Department of Mathematics \& Statistics at McGill University for their hospitality during the initial stages of this article.

\section*{Appendix: Computational Comments}

We used PARI/GP for the calculations leading to Corollary~\ref{c:numerical_bounds} and Corollary~\ref{cor:sym-numbers}. 
To manipulate derivatives of Bessel functions, we used the standard recurrence formula
\[
J_n'(r) = \frac{J_{n-1}(r) - J_{n+1}(r) }{2}
\]
found in \cite[\S 2.12 formula (2)]{W} for instance.
For a bounded range of $r$, $J_n(r)$ decays rapidly as $n$ increases, roughly as $(r/2)^n/n!$.
In particular, $J_n(r)$ will achieve its maximum at the largest value of $r$ in range. 
Explicit inequalities of this type \cite[\S 2.11 formula (5)]{W} also allow one to bound the error in truncating the sum $S$ to finitely many terms. 

To numerically evaluate the integral
\[
\int_{ZS/\varepsilon}^{\infty} \left( 1 - \frac{2ZS}{\varepsilon x} \right) e^{-x^2/2} dx
\]
we expressed it in terms of the incomplete Gamma function
\[
\Gamma(s,x) = \int_x^{\infty} t^{s-1} e^{-t} dt.
\]

The symmetrization method can be implemented as follows. One inputs a vector \texttt{r} listing the radii $r_1, \ldots, r_M$. Then \texttt{mu1(r)} is the bound obtained from Theorem~\ref{thm:sym}.
The function \texttt{b(r,x)} is used to write the integrand in Theorem~\ref{thm:sym} and \texttt{q(r)} computes the necessary integrals using the incomplete gamma function. 
To find the optimal $T$, the program solves (\ref{e:T-optimal}) numerically over $T$ less than the quantity \texttt{Tmax(r)}. There might be no solution for $T$ if \texttt{Tmax} is too small, but once large enough, its precise value does not affect the results. Our choice of \texttt{Tmax} is made in view of the input leading to Corollary~\ref{cor:sym-numbers}. In that case, $M=5$ and among $r_1, \ldots, r_5$, it is $r_5$ that comes closest to a zero of $J_0(r)$. 
The choice below automatically increases \texttt{Tmax} as $r_5$ approaches $j_{0,2}$.\\
\noindent \\
\noindent \texttt{ \noindent
b(r,x)=(1/sqrt(2))*r*exp(-x\string^2*besselj(0,r)\string^2/2/(1-besselj(0,r)\string^2));\\
Tmax(r)=5/abs(besselj(0,r[5]));\\
T(r)=solve(x=0,Tmax(r),1-sum(i=1,length(r),b(r[i],x)));\\
q(r)=incgam(1/2,T(r)\string^2/2)- sqrt(1/2)*sum(i=1,length(r), \\
\indent \indent \indent r[i]*incgam(1/2,T(r)\string^2/(2*(1-besselj(0,r[i])\string^2))));\\
mu1(r)=sqrt(Pi)*q(r)/(r[length(r)]\string^2);\\
}

Here are the commands used to obtain the bound for $\mu(1)$ stated in Corollary~\ref{c:numerical_bounds}.
 The first steps set $\delta=1/2$ and evaluate the first three roots, as well as the first two critical points, of $J_0$.\\
\noindent
\begin{footnotesize}
\texttt{del=0.5;\\
j1=solve(x=2,3,besselj(0,x));\\
j2=solve(x=5,6,besselj(0,x));\\
j3=solve(x=8,9,besselj(0,x));\\
c1=solve(x=3,4,besselj(1,x));\\
c2=solve(x=5,8,besselj(1,x));\\
}
\end{footnotesize}

The next steps find the largest $\varepsilon$ allowed by Proposition~\ref{prop:stay1} for the given value of $\delta$. Then the endpoints $a_k(\varepsilon)$, $b_k(\varepsilon)$ are computed by solving $J_0(r) = \pm \varepsilon$. \\
\begin{footnotesize}
\texttt{
eps0=abs(besselj(0,c2));\\
eps1=(del/(1+del))*sqrt(1-(del/j1)\textasciicircum 2)*abs(besselj(1,j1+del))/sqrt(2);\\
eps2=(del/(1+del))*sqrt(1-(del/j2)\string^2)*abs(besselj(1,j2+del))/sqrt(2);\\
eps=min(min(eps0,eps1),eps2);\\
a1=solve(x=0,j1,besselj(0,x)-eps);\\
b1=solve(x=j1,c1,besselj(0,x)+eps);\\
a2=solve(x=c1,j2,besselj(0,x)+eps);\\
b2=solve(x=j2,c2,besselj(0,x)-eps);\\
a3=solve(x=c2,j3,besselj(0,x)-eps);\\
}
\end{footnotesize}

By inspecting plots, one determines which values of $n$ lead to maxima for $J_n(r), J_n'(r), J_n(r)/r$ somewhere other than $j_{0,1}+\delta$. The next commands calculate where these maxima occur, as reported in Table~\ref{table:contributions}.\\
\noindent
\begin{footnotesize}
\texttt{
u1=j1-del;\\
v1=solve(x=j1-del,j2+del,0.5*(besselj(1-1,x)-besselj(1+1,x))/x-besselj(1,x)/(x\string^2));\\
w1=solve(x=j1-del,j2+del,0.75*besselj(1,x)-0.25*besselj(3,x));\\
u2=solve(x=j1-del,j2+del,besselj(2-1,x)-besselj(2+1,x));\\
v2=solve(x=j1-del,j2+del,0.5*(besselj(2-1,x)-besselj(2+1,x))/x-besselj(2,x)/(x\string^2));\\
w2=solve(x=j1-del,j2+del,besselj(2-2,x)-2*besselj(2,x)+besselj(2+2,x));\\
u3=solve(x=j1-del,j2+del,besselj(3-1,x)-besselj(3+1,x));\\
v3=solve(x=j1-del,j2+del,0.5*(besselj(3-1,x)-besselj(3+1,x))/x-besselj(3,x)/(x\string^2));\\
w3=j2+del;\\
u4=solve(x=j1-del,j2+del,besselj(4-1,x)-besselj(4+1,x));\\
v4=solve(x=j1-del,j2+del,0.5*(besselj(4-1,x)-besselj(4+1,x))/x-besselj(4,x)/(x\string^2));\\
w4=solve(x=j1-del,j2+del,besselj(4-2,x)-2*besselj(4,x)+besselj(4+2,x));\\
u5=j2+del;\\
v5=solve(x=j1-del,j2+del,0.5*(besselj(5-1,x)-besselj(5+1,x))/x-besselj(5,x)/(x\string^2));\\
w5=solve(x=j1-del,j2+del,besselj(5-2,x)-2*besselj(5,x)+besselj(5+2,x));\\
u6=j2+del;\\
v6=j2+del;\\
w6=solve(x=j1-del,j2+del,besselj(6-2,x)-2*besselj(6,x)+besselj(6+2,x));\\
}
\end{footnotesize}

The contributions to $S$ from $n=1,\ldots,6$ and $n \geq 7$ are computed as follows.\\
\begin{footnotesize}
\noindent
\texttt{
S1=abs(besselj(1,u1))+abs(besselj(1,v1))/v1+0.5*abs(besselj(0,w1)-besselj(2,w1));\\
S2=abs(besselj(2,u2))+abs(besselj(2,v2))/v2+0.5*abs(besselj(1,w2)-besselj(3,w2));\\
S3=abs(besselj(3,u3))+abs(besselj(3,v3))/v3+0.5*abs(besselj(2,w3)-besselj(4,w3));\\
S4=abs(besselj(4,u4))+abs(besselj(4,v4))/v4+0.5*abs(besselj(3,w4)-besselj(5,w4));\\
S5=abs(besselj(5,u5))+abs(besselj(5,v5))/v5+0.5*abs(besselj(4,w5)-besselj(6,w5));\\
S6=abs(besselj(6,u6))+abs(besselj(6,v6))/v6+0.5*abs(besselj(5,w6)-besselj(7,w6));\\
M=100;\\
Stail=sum(n=7,M,\\
\indent \indent \indent \indent (1+n/(j2+del))*abs(besselj(n,j2+del))+\\
\indent \indent \indent \indent 0.5*abs(besselj(n-1,j2+del)-besselj(n+1,j2+del)));\\
S=S1+S2+S3+S4+S5+S6+Stail;\\
}
\end{footnotesize}

Finally, with $S$ in hand, the integral is computed using the incomplete gamma function.\\
\begin{footnotesize}
\noindent
\texttt{
tmin=0.5*Pi*(S/eps)\string^2;\\
int = (1/sqrt(2))*incgam(1/2,tmin) - (sqrt(Pi)/2)*(S/eps)*incgam(-1/2,tmin);
}
\end{footnotesize}


\begin{thebibliography}{99}

\bibitem{A16} N. Anantharaman, \emph{Topologie des hypersurfaces nodales de fonction al\'eatories gaussiennes}, S\'eminaire Bourbaki, no. 1116 (2016).

\bibitem{Bar} A.~H. Barnett, \url{https://users.flatironinstitute.org/~ahb/rpws/}, Simons Flatiron Institute (2020) 

\bibitem{BarJin} A.~H. Barnett and M. Jin, \emph{Statistics of random plane waves}, in preparation.

\bibitem{Berry77} M. Berry, \emph{Regular and irregular semiclassical wavefunctions}, J. Phys. A: Math. General 10(12), 2083 (1977).

\bibitem{BCW19} D. Beliaev, V. Cammorota, and I. Wigman, \emph{Two point function for critical points of a random plane wave}, Int. Math. Res. Not., Volume 2019, Issue 9 2661-2689 (2019).

\bibitem{BK13} D. Beliaev and Z. Kereta, \emph{On the Bogomolny-Schmit Conjecture}, J. Phys. A: Math. The., 46 (45) (2013).

\bibitem{BS02} E. Bogomolny and C. Schmit, \emph{Percolation model for nodal domains of chaotic wave functions}, Phys. Rev. Lett. 88, 114102 (2002).

\bibitem{BL13} N.~Burq and G.~Lebeau, \emph{Injections de Sobolev probabilistes et applications.}  Ann. Sci. \'{E}c. Norm. Sup\'{e}r. (4), 46 (2013), 917–-962. arXiv:1111.7310. (2011)

\bibitem{CHsup} Y.~Canzani and B.~Hanin. \emph{High Frequency Eigenfunction Immersions and Supremum Norms of Random Waves.} Electronic Research Announcements in Mathematical Sciences, Volume 22, 2015, pp. 76-86. arXiv: 1406.2309.

\bibitem{CH15} Canzani, Y. and B. Hanin. \emph{Scaling limit for the kernel of the spectral projector and remainder estimates in the pointwise Weyl law.} Anal. PDE 8, no. 7 (2015): 1707-1732.

\bibitem{CS18} Y. Canzani and P. Sarnak, \emph{Topology and nesting of the zero set component of monochromatic random waves}, Comm. Pure. Appl. Math. (2018).

\bibitem{CKM} H.~Cohn, A.~Kumar, and G.~Minton, \emph{Optimal simplices and codes in projective spaces} Geometry and Topology 20 (2016), 1289-1357, arXiv:1308.3188

\bibitem{D67} R.~M.~Dudley, \emph{The sizes of compact subsets of Hilbert space and the continuity of Gaussian processes.} J. Functional Analysis {\bf 1}, 290-330 (1967)

\bibitem{EPS13} A.~Enciso and D.~Peralta-Salas, \emph{Submanifolds that are level sets of a solution to an elliptic PDE}, Adv. Math. 249 (2013) 204–249. arXiv: 1007.5181

\bibitem{GW14} D. Gayet and J.Y. Welschinger,  \emph{Expected topology of random real algebraic submanifolds}, J. Lond. Math. Soc. (2) 90(1), 105–-120 (2014).

\bibitem{GW15} D. Gayet and J.Y. Welschinger,  \emph{Lower estimates for the expected Betti numbers of random real hypersurfaces}, J. Inst. Math. Jussieu (2015) 14(4), 673-702.

\bibitem{IR19} M.~Ingremeau and A.~Rivera, \emph{A lower bound for the Bogomolny-Schmit constant for random monochromatic plane waves} arXiv:1803.02228 to appear in Mathematics Research Letters (2019)

\bibitem{K12} K. Konrad, \emph{Asymptotic Statistics of Nodal Domains of Quantum Chaotic Billiards in the Semiclassical Limit}, Senior Thesis, Dartmouth College (2012). Available at \url{https://users.flatironinstitute.org/~ahb/dartmouth/papers/Konrad2012thesis.pdf}

\bibitem{LL15} A. Lerario and E. Lundberg, \emph{Statistics on Hilbert’s 16th Problem},  Int. Math. Res. Not., Vol. 2015, No. 12, pp. 4293–4321 (2015) .

\bibitem{N11} M. Nastasescu, \emph{The number of ovals of a real plane curve}, Senior Thesis, Princeton University (2011).

\bibitem{NS2009}  F.~Nazarov and M.~Sodin, \emph{On the number of nodal domains of random spherical harmonics}, Amer. J. Math. 131(5), 1337–-57 (2009)

\bibitem{NS2016} F.~Nazarov and M.~Sodin, \emph{Asymptotic laws for the spatial distribution and the number of connected components of zero sets of Gaussian random functions}, Zh. Mat. Fiz. Anal. Geom., 2016, Volume 12, Number 3, 205–-278 DOI: https://doi.org/10.15407/mag12.03.205

\bibitem{N} J.~W.~Neuberger, \emph{The continuous Newton's method, inverse functions, and Nash-Moser} Amer. Math. Monthly Vol. 114, No. 5 (May, 2007), pp. 432--437

\bibitem{Nic15} L. Nicolaescu, \emph{Critical sets of random smooth functions on compact manifolds}, Asian J. Math. Vol. 19, No. 3 (2015).

\bibitem{So12} M. Sodin, \emph{Lectures on random nodal portraits}, lecture notes for a mini-course given at the St. Petersburg Summer School in Probability and Statistical Physics, \url{http://www.math.tau.ac.il/~sodin/SPB-Lecture-Notes.pdf} (2012)

\bibitem{SW18} P. Sarnak and I. Wigman, \emph{Topologies of nodal sets of random band-limited functions}, Comm. Pure. Appl. Math. (2018).

\bibitem{TL71} H.~N.~V.~Temperley and E.~H. Lieb, \emph{Relations between the 'Percolation' and 'Colouring' Problem and other Graph-TheoreticalProblems Associated with Regular Planar Lattices: Some Exact Results for the 'Percolation'Problem}, Proceedings of the Royal Society of London. Series A, Mathematical and Physical Sciences, Vol. 322, No. 1549 (Apr. 20, 1971), pp. 251-280

\bibitem{W} G.~N.~Watson, A Treatise on the Theory of Bessel Functions. Second edition. Cambridge University Press, Cambridge, 1958.

\bibitem{ZFA} R. Ziff, S. Finch, and V. Adamchik, \emph{Universality of finite-size corrections to the number of critical percolation clusters}, Phys. Rev. Letters, 79 (18) (1997).

\end{thebibliography}
\end{document}